\documentclass[10pt]{article}
\usepackage{amssymb,amsmath,amsthm}
\usepackage{enumerate}
\usepackage{tikz}
\usepackage{hyperref}
\usetikzlibrary{arrows}

\def\l{{\lambda}}
\def\Z{{\mathbb{Z}}}

\newenvironment{packedItem}{
\begin{itemize}
  \setlength{\itemsep}{1pt}
  \setlength{\parskip}{0pt}
  \setlength{\parsep}{0pt}
}{\end{itemize}}

\newenvironment{packedEnum}{
\begin{enumerate}
  \setlength{\itemsep}{1pt}
  \setlength{\parskip}{0pt}
  \setlength{\parsep}{0pt}
}{\end{enumerate}}

\let\oldmarginpar\marginpar
\renewcommand\marginpar[1]{\-\oldmarginpar[\raggedleft\footnotesize #1]%
{\raggedright\footnotesize #1}}

\newtheorem{theorem}{Theorem}
\newtheorem{corollary}[theorem]{Corollary}
\newtheorem{lemma}[theorem]{Lemma}

\newtheorem{observation}[theorem]{Observation}

\theoremstyle{definition}
\newtheorem{construct}{Construction}

\renewcommand{\a}{\alpha}
\renewcommand{\b}{\beta}
\newcommand{\g}{\gamma}

\newcommand{\TS}{{\rm{TS}}}
\newcommand{\ts}{{\rm{TS}}}
\newcommand{\BIBD}{{\rm{BIBD}}}

\newcommand{\etour}{\mathbb{E}}

\newcommand{\s}{\mathcal{S}}

\usepackage{enumitem}
\setlistdepth{9}

\setlist[enumerate,1]{label=\arabic*.}
\setlist[enumerate,2]{label=(\alph*)}
\setlist[enumerate,3]{label=\roman*.}
\setlist[enumerate,4]{label=\Alph*.}
\setlist[enumerate,5]{label=\Roman*.}
\setlist[enumerate,6]{label=(\arabic*)}
\setlist[enumerate,7]{label=\alph*.}
\setlist[enumerate,8]{label=(\Alph*)}
\setlist[enumerate,9]{label=(\Roman*)}

\renewlist{enumerate}{enumerate}{9}

\newcommand{\vanish}[1]{}

\begin{document}

 \title{${\rm{TS}}(v,\lambda)$ with cyclic 2-intersecting Gray codes:\\
 $v\equiv 0$ or $4\pmod{12}$}

 \author{
 John Asplund\\
 {\small Department of Technology and Mathematics} \\
 {\small Dalton State College} \\
 {\small Dalton, GA 30720, USA} \\
 {\small jasplund@daltonstate.edu}\\
 \\
 Melissa Keranen\\
  {\small Department of Mathematical Sciences} \\
 {\small Michigan Technological University} \\
 {\small Houghton, MI 49931, USA} \\
 {\small msjukuri@mtu.edu}\\
\\
  }
 \date{}
 \maketitle

\begin{abstract}
A ${\rm{TS}}(v,\l)$ is a pair $(V,\mathcal{B})$ where $V$ contains $v$ points and $\mathcal{B}$ contains $3$-element subsets of $V$ so that each pair in $V$ appears in exactly $\l$ blocks. A $2$-block intersection graph ($2$-BIG) of a ${\rm{TS}}(v,\l)$ is a graph where each vertex is represented by a block from the ${\rm{TS}}(v,\l)$ and each pair of blocks $B_i,B_j\in \mathcal{B}$ are joined by an edge if $|B_i\cap B_j|=2$. 
We show that there exists a ${\rm{TS}}(v,\l)$ for $v\equiv 0$ or $4\pmod{12}$ whose $2$-BIG is Hamiltonian. 
This is equivalent to the existence of a $\TS(v,\l)$ with a cyclic 2-intersecting Gray code.
\end{abstract}

\section{Introduction}
A balanced incomplete block design $\BIBD(v,k,\lambda)$  is a set of $v$ points with a collection of blocks of size $k$ with the property that every pair of points appears in exactly $\lambda$ blocks. The design is said to be \textit{simple} if there are no repeated blocks.  If we require the blocks to have size $k=3$, then we call it a triple system, $\ts(v,\lambda)$. A $\ts(v,1)$ is called a Steiner triple system. 
Dehon \cite{D} showed that there exists a simple $\ts(v,\l)$ if and only if $\l\leq v-2$, $\l v(v-1)\equiv 0\pmod{6}$ and $\l(v-1)\equiv 0\pmod{2}$. We say that $v,\l$ are \textit{admissible} if there exists a $\ts(v,\l)$. 

A \textit{block intersection graph} (BIG) of a $\ts(v,\l)$ $(V,\mathcal{B})$ is a graph where each block in $\mathcal{B}$ represents a vertex and two vertices in the BIG are adjacent if the corresponding blocks share at least one point. 
This can be extended by varying the adjacency rules as follows. A \textit{$k$-block intersection graph} ($k$-BIG) of a $\ts(v,\l)$ $(V,\mathcal{B})$ has the same vertex set as a BIG, but two vertices in the $k$-BIG are adjacent if they share \textit{exactly} $k$ vertices in common. 
A cycle in a graph that contains all the vertices in it is called a {\em Hamilton cycle}, and if a graph contains a Hamilton cycle we say it is {\em Hamiltonian}.  Similarly, a path in a graph that contains all the vertices in it is called a {\em Hamilton path}. The question of hamiltonicity in $k$-BIGs is closely related to the study of Gray codes.

An $n$-bit Gray code, is an ordering of the $2^{n}$ strings of length $n$ over $\{0,1\}$ such that every pair of successive strings differ in exactly one position. Frank Gray, a physicist and researcher at Bell Labs, introduced the use of these codes in 1947 to prevent spurious output from electromechanical switches. In recent years, these codes have been useful in the fields of error correcting codes and communication.
A generalization of this concept is the term {\em combinatorial Gray code} which was introduced in 1980 and refers to any method for generating combinatorial objects so that successive objects differ in some pre-described, small way \cite{CS}.
A $\kappa$-intersecting Gray code for a $\BIBD(v,k,\lambda)$, $S$, is a listing of the blocks of $S$ in such a way that consecutive blocks
intersect in exactly $\kappa$ points.  Thus the existence of $\kappa$-intersecting Gray codes for $\BIBD(v,k,\lambda)$s where $1 \leq \kappa \leq k-1$ can be established by determining if the $\kappa$-BIG graph of the $\BIBD$ contains a Hamilton path.  A
{\em cyclic $\kappa$-intersecting Gray code} is equivalent to a Hamilton cycle in the $\kappa$-BIG.

Our investigation focuses on the existence of cyclic $2$-intersecting Gray codes for $\ts(v,\l)$s.
Dewar \cite{dewar} showed that there exists a $\ts(v,2)$ whose $2$-BIG is Hamiltonian if $v\equiv 3$ or $7\pmod{12}$ and $v\geq 7$, or $v\equiv 1$ or $4\pmod{12}$ and $v\not\equiv 0\pmod{5}$. 
This result was extended by Erzurumluo\u{g}lu and Pike in \cite{aras2} where the complete spectrum was given for the existence of $\ts(v,2)$s whose $2$-BIGs are Hamiltonian. 
It is important to realize that there can be many distinct $\ts(v,\l)$ with the same parameters. Thus it is possible for there to exist a $\ts(v,\l)$ that is Hamiltonian (as shown in this paper) and for a $\ts(v,\l)$ to be non-Hamiltonian (as shown in \cite{aras1}).
On the other hand, Hor\'{a}k \cite{horak} showed that the BIG of all $\ts(v,1)$ are Hamiltonian. 
Later, Alspach et al. \cite{AHM} showed that under certain conditions all pairwise balanced designs with the same parameters have a Hamiltonian BIG.
In \cite{HPR}, it was shown that all $\ts(v,\l)$s have Hamiltonian $1$-BIGs for arbitrary index $\l$.
Despite the fact that the $1$-BIG of a $\ts(v,\l)$ and the $2$-BIG of the same $\ts(v,\l)$ are subgraphs of the BIG of the same $\ts(v,\l)$, not all $\ts(v,\l)$s have Hamiltonian $2$-BIGs. In \cite{aras1}, all but a finite number of elements of the spectrum were determined for which there exists a $\ts(v,2)$ whose $2$-BIG is connected but has no Hamilton path (and therefore no Hamilton cycle). 
Mahmoodian made the observation that the $2$-BIG of the unique $\ts(6,2)$ is the Petersen graph \cite{mah}. Additionally, Colbourn and Johnstone \cite{CJ} showed that there is a $\ts(19,2)$ whose $2$-BIGs is connected but is not Hamiltonian. While the existence question is settled in the case of $1$-BIGs for $\ts(v,\l)$ with arbitrary index $\l$, this is the first paper to date that focuses on the case where $\l>2$ when $\kappa=2$.

As $\lambda$ increases, the $2$-BIG for the corresponding triple system contains more edges, so one might suspect that it becomes more likely that a Hamilton cycle would exist. However, because the $2$-BIG gets significantly more complicated as $\lambda$ increases, it quickly becomes difficult to determine the structure of the graph.
Recall that a $\ts(v,\l)$ exists if and only if $\l\leq v-2$, $\l v(v-1)\equiv 0\pmod{6}$ and $\l(v-1)\equiv 0\pmod{2}$. Thus when $v$ is even, $\l$ must be even. Suppose $n \equiv 1,5 \pmod{6}$ and $v=2n+2$. 
In \cite{S}, Schreiber gives a construction (see Section~\ref{sec:construction}) for a set of $\binom{v}{3}$ triples (blocks) on $v$ points that can be partitioned into $n$ sets, each forming a $\ts(v,2)$. Then the union of any $t$ of these sets is a simple $\ts(v,2t)$ for $1 \leq t \leq n$.
In this paper, we show that the $2$-BIG of any $\ts(v,\l)$ formed by taking $t=\frac{\l}{2}$ sets from the Schreiber construction is Hamiltonian. Thus the main result is as follows.

\begin{theorem}\label{v04mod12}
If $v\equiv 0$ or $4\pmod{12}$, then for all admissible $v,\l$ there exists a simple $\TS(v,\l)$ with a cyclic $2$-intersecting Gray code.
\end{theorem}

\section{The Construction}\label{sec:construction}
Let $n\equiv 1$ or $5\pmod{6}$ and $v=2n+2$ for the remainder of the paper.
The following construction is a restatement of what was given by Schreiber in \cite{S} for forming $v \choose 3$ triples on $v$ points that can be partitioned into $n$ sets, each forming a $\ts(v,2)$.  Let $\Z_n \times \Z_2 \cup \{\infty_1,\infty_2\}$ be the set of $v$ points.

\begin{construct}\label{sblocks}
\cite{S}
Let $g\in \Z_n=\{0,1,\ldots,n-1\}$, and let $\tau$ denote the automorphism of $\Z_n$ such that $\tau(h)=-2h$ for each $h\in \Z_{n}$. Then $\tau(0)=0$, and since $\Z_n$ can contain no element of order $3$ (as $\gcd(n,3)=1$), $\tau(h)\neq h$ for $h\neq 0 $. Thus $\tau$ permutes the elements of $\Z_n^*$ ($\Z_n^*=\Z_n\setminus\{0\}$) into cycles. Call the ordered pair $(h,\tau(h))$, $g\in \Z_n^*$, an \textit{arc} and $(-h,\tau(-h))$ the \textit{opposite} arc in a directed graph on the vertex set $\Z_n$. Thus there are $\frac{1}{2}(n-1)$ pairs of opposite arcs for each $h$. 
Let $D_g'$ be the directed graph with vertex set $\Z_n$ and formed by the set of arcs in $\{(h+g,\tau(h)+g),(-h+g,\tau(-h)+g)\,:\, h\in \Z_n\}$. We form the set of blocks $\s_g$ for each $g\in \Z_n$ in the following manner:\\

\textit{Step $0$.} For each pair of opposite arcs (on the same cycle or not) arbitrarily color one red, one blue.\\

\textit{Step $1$a.} Take each triple of different elements $\{(a+g,0),(b+g,0),(c+g,0)\}$, where $a,b,c\in \Z_n$, if $a+b+c\equiv 3g\pmod{n}$.\\

\textit{Step $1$b.} Suppose $\{(a,0),(b,0),(c,0)\}$ is a triple from Step $1$a. Form $7$ more triples:\\
$\{(a,1),(b,0),(c,0)\}$, $\{(a,0),(b,1),(c,0)\}$, $\{(a,0),(b,0),(c,1)\}$, $\{(a,1),(b,1),(c,0)\}$, $\{(a,1),(b,0),(c,1)\}$, \\
$\{(a,0),(b,1),(c,1)\}$, $\{(a,1),(b,1),(c,1)\}$.\\

\textit{Step $1$c.} For each arc $(a,b)$ (that is , $b=-2a$), we add the two triples $\{(a,0),(b,0),(a,1)\}$ and $\{(a,0),(b,1),(a,1)\}$.\\

\textit{Step $2$a.} If $(a,b)$ is a \textit{red} arc, add the $4$ triples $\{\infty_0,(a,0),(b,0)\}$, $\{\infty_0,(a,1),(b,1)\}$, $\{\infty_1,(a,0),(b,1)\}$, and $\{\infty_1,(a,1),(b,0)\}$. \\

\textit{Step $2$b.} If $(a,b)$ is a \textit{blue} arc, add the $4$ triples $\{\infty_1,(a,0),(b,0)\}$, $\{\infty_1,(a,1),(b,1)\}$, $\{\infty_0,(a,0),(b,1)\}$, and $\{\infty_0,(a,1),(b,0)\}$. \\

\textit{Step $3$.} Add the $4$ triples $\{\infty_0,\infty_1,(g,0)\}$, $\{\infty_0,\infty_1,(g,1)\}$, $\{\infty_0,(g,0),(g,1)\}$ and $\{\infty_1,(g,0),(g,1)\}$. \\
\end{construct}

Based on this construction, for each $g\in \mathbb{Z}_n$, $\s_g$ is a set of triples that forms a $\ts(v,2)$ and $\s_{g_1}\cap \s_{g_2}=\varnothing$ for each $g_1,g_2\in \Z_n$. 
Thus, $\bigcup_{i=1}^{\l/2} \s_i$ is a set of triples that forms a $\ts(v,\l)$.
Furthermore, the union of any $t$ of these $S_{g}$'s is a simple $\TS(v,2t)$ for $1 \leq t \leq n$.

\section{Subgraph of $2$-BIG Corresponding to Steps 1a and 1b}\label{sec:v04mod6}
To show that the $2$-BIG of a $\ts(v,\l)$ formed by Construction~\ref{sblocks} is Hamiltonian, we will find vertex disjoint paths that together span the vertex set and whose end vertices are adjacent.  Then we will concatenate the paths to form a Hamilton cycle.   
We begin this section by finding a Hamilton cycle through the portion of the $2$-BIG that is formed from the blocks of the $\ts(v,\l)$ that are given in Steps 1a and 1b of Construction~\ref{sblocks}.

 A lobster graph is a tree in which all the vertices
are within distance two of a central path. We call a path with $k$ edges $\{a_0a_1,a_1a_2,\ldots,a_{k-1}a_k\}$ a {\em $k$-path} and is denoted as $[a_0,a_1,\ldots,a_k]$. 
We denote a walk on the edges $\{x_1x_2,x_2x_3, \ldots, x_{k-1}x_{k}\}$ as 
$(x_1,x_2,x_3, \ldots,x_k)$. Let $W$ be a walk on a graph $G$ that visits every vertex.  
We say that $W=(x_1,x_2,\ldots,x_k)$ is an {\em $(x_1,x_k)$-Hamilton walk} in $G$ if 
the subgraph induced by the edges (multiedges removed) of this walk is spanning and a lobster graph.

Recall that $n\equiv 1$ or $5\pmod{6}$ and $v=2n+2$. Consequently, since $v$ is even, $\l$ is also even. 
Let $G$ be the $2$-BIG of the $\ts(v,\l)$ formed from Construction~\ref{sblocks} and let $H^n_{g_{1},g_{2},\dots,g_{k}}$ be the subgraph of $G$ induced by the triples
given in Step 1a of Construction~\ref{sblocks} using $\s_{g_{1}}$, $\s_{g_{2}}$ $\cdots$ $\s_{g_{k}}$.
Consider how the triples (blocks) are described in Step~1a.
Because the second coordinate of each point in the triples described in Step $1a$ is $0$, we will let $\{\a,\b,\g\}_{g}$ denote the triple $\{(a+g,0), (b+g,0), (c+g,0)\}$ where $\a=a+g$, $\b=b+g$, $\g=c+g$ for some $g\in \Z_n$.
For $g_1,g_2\in\Z_n$,  $H^n_{g_{1},g_{2}}$  can be drawn in a honeycomb configuration when $\gcd\{g_1,g_2,n\}=1$, so we will refer to this subgraph as the {\em honeycomb graph}. We provide the honeycomb graphs $H_{0,1}^{11}$ and $H_{0,1}^{13}$ in Figures~\ref{hex11} and \ref{hex13} respectively.  

\begin{figure}
 \begin{minipage}{.45\textwidth}
  \centering
	\mbox{}
  \vspace{.5in}
	\includegraphics[scale=1]{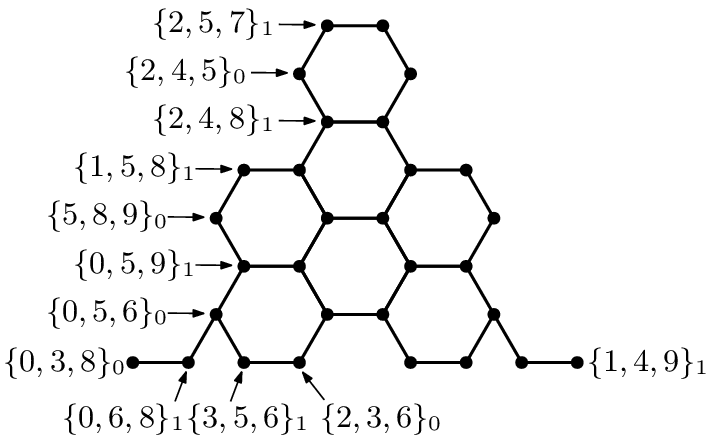}
\caption{$2$-BIG formed from blocks in Step $1a$ in Construction~\ref{sblocks} when $n=11$, $g_1=0$, and $g_2=1$}\label{hex11}
 \end{minipage}
\hspace{.2in}
 \begin{minipage}{.45\textwidth}
  \centering
	\includegraphics[scale=1]{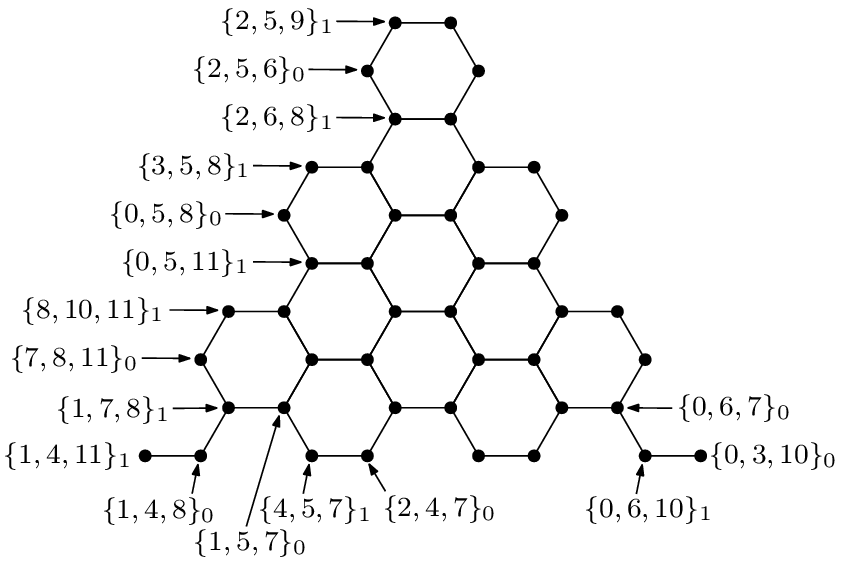}
\caption{$2$-BIG formed from blocks in Step $1a$ in Construction~\ref{sblocks} when $n=13$, $g_1=0$, and $g_2=1$}\label{hex13}
 \end{minipage}
\end{figure}

\begin{lemma}\label{2hexagons}
$H^n_{0,1, \ldots,\frac{\l}{2}-1}$ contains a $(Z,Y)$-Hamilton walk where $Z=\{n-2,1,4\}$ and $Y=\{\frac{\l}{2}+1,\frac{\l}{2}-2,\frac{\l}{2}-5\}$ if $\frac{\l}{2}$ is even, or $Y=\{\frac{\l}{2},\frac{\l}{2}-3,\frac{\l}{2}-6\}$ if $\frac{\l}{2}$ is odd.
\end{lemma}


\begin{proof}
Let $Z_1, Z_2 \in V(H_{g_1,g_2,\ldots g_k}^{n})$ where $Z_1=\{3g_1-2g_2,g_2,-3g_1+4g_2\}$ and $Z_2=\{-2g_1+3g_2,g_1,4g_1-3g_2\}$. Suppose we found a $(Z_1,Z_2)$-Hamilton walk through each of the following subgraphs of $H^n_{0,1, \ldots,\frac{\l}{2}-1}$:

\[H^n_{g-1,g}\mbox{ for each }g \in \{1,3,\ldots,\frac{\l}{2}-1\} \;\;\;\;\; \text{if $\frac{\l}{2}$ is even}\]
or
\[
H^n_{g-1,g}\mbox{ for each }g \in \{1,3,\ldots,\frac{\l}{2}-4\} \mbox{ and } H^n_{\frac{\l}{2}-3,\frac{\l}{2}-2,\frac{\l}{2}-1}\;\;\;\;\; \text{if $\frac{\l}{2}$ is odd}.
\]

Note that if $Z_2 \in V(H^n_{g-1,g})$ and $Z_1 \in V(H^n_{g+1,g+2})$ then $Z_2$ is adjacent to $Z_1$ in $H^n_{0,1, \ldots,\frac{\l}{2}-1}$.
Let $W^{n}_{g_1,g_2, \ldots, g_k}$ be the subgraph of $H^n_{g_1,g_2, \ldots, g_k}$ induced by the edges of the Hamilton walk through $H^n_{g_1,g_2,\ldots g_k}$. Then we obtain a $(Z,Y)$-Hamilton walk through $H^n_{0,1, \ldots,\frac{\l}{2}-1}$ by concatenating the following walks. 

\[
W^{n}_{0,1,\ldots, \frac{\l}{2}-1}=\begin{cases}
W^{n}_{0,1} \circ W^n_{2,3} \circ \dots \circ W^n_{\frac{\l}{2}-2,\frac{\l}{2}-1} \mbox{ if } \l \mbox{ is even}\\
W^n_{0,1} \circ W^n_{2,3} \circ \dots \circ W^n_{\frac{\l}{2}-5,\frac{\l}{2}-4} \circ  W^n_{\frac{\l}{2}-3,\frac{\l}{2}-2,\frac{\l}{2}-1} \mbox{ if } \l \mbox{ is odd}\\
\end{cases}
\]

We proceed by finding a $(Z_1,Z_2)$-Hamilton walk $W^n_{g-1,g}$ through the honeycomb graph $H^n_{g-1,g}$ by induction on $n$.
For $n=11$ and any $g$, the corresponding honeycomb graph $H^{11}_{g-1,g}$ is isomorphic to the one given in Figure~\ref{hex11path}. Label the vertices in this honeycomb graph as in Figure~\ref{hex11path}. 
\begin{figure}
 \begin{minipage}{.45\textwidth}
  \centering
	\mbox{}
  \vspace{.3in}
	\includegraphics[scale=1]{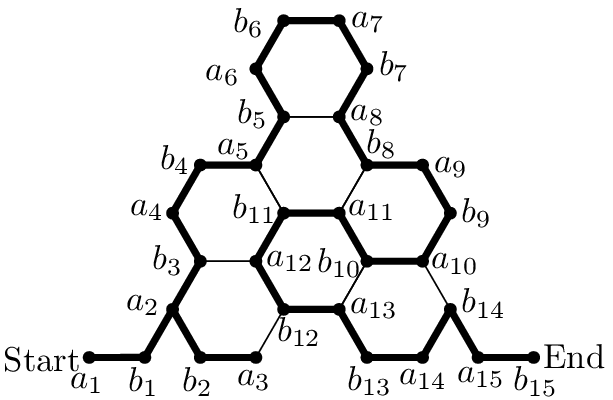}
\caption{Walk used in $2$-BIG formed from blocks in Step $1a$ in Construction~\ref{sblocks} when $n=11$}\label{hex11path}
 \end{minipage}
\hspace{.2in}
 \begin{minipage}{.45\textwidth}
  \centering
	\includegraphics[scale=1]{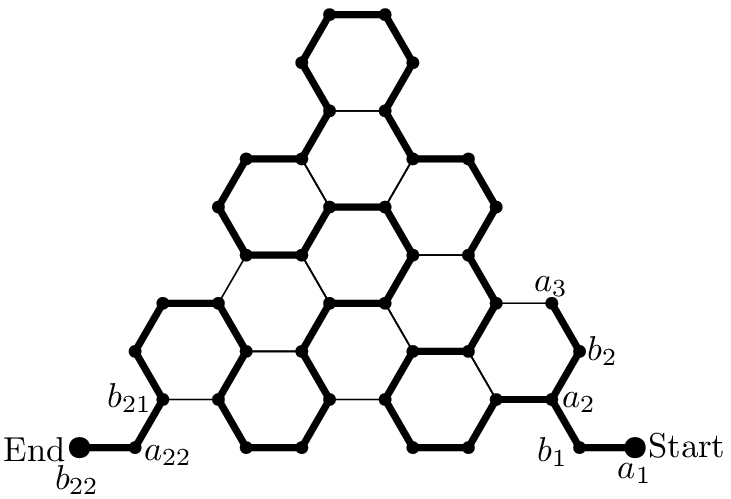}
\caption{Walk used in $2$-BIG formed from blocks in Step $1a$ in Construction~\ref{sblocks} when $n=13$}\label{hex13path}
 \end{minipage}
\end{figure}
Referring to Figure~\ref{hex11path}, define a walk on the honeycomb graph by 
\[W^{11}_{g-1,g}=(a_1,b_1,a_2,b_2,a_3,b_2,a_2,b_3,a_4,b_4,\ldots,a_{15},b_{15}).\]  
It is clear that the subgraph induced by the vertices of this walk is a lobster graph. Thus 
$W^{11}_{g-1,g}$ is a Hamilton walk. Figure~\ref{hex13path} gives the Hamilton walk, $H^{13}_{g-1,g}$.

By induction we may assume that for $v=2k+2$ where $k\equiv 1$ or $5\pmod{6}$ and $k<n$, there exists
a $(Z_1,Z_2)$-Hamilton walk through the honeycomb graph $H^k_{g-1,g}$ for all $g\in\{1,3,5,\ldots,\frac{\lambda}{2}-1\}$. 
First, we will show that $H^{n-6}_{g-1,g}$ (which has a Hamilton walk by the induction hypothesis) is a subgraph of $H^n_{g-1,g}$. 
Then we will show that $H^n_{g-1,g}$ has a $(Z_1,Z_2)$-Hamilton walk.

We denote  the vertices of $H^n_{g-1,g}$ by $U_{i,j}$. 
Recall that each vertex in $H^n_{g-1,g}$ is a triple from Step $1a$ of
Construction~\ref{sblocks}; we define these vertices as follows.
For each $i\in\{0,1,\ldots,(n-7)/3\}$ and each $j\in\{1,2,\ldots,3i+4\}$, 
when $n\equiv 1\pmod{6}$ let
\[
U_{i,j}=\begin{cases}
\left\{g+3i+5,\frac{1}{2}(2g-3j+6i+7),\frac{1}{2}(2g+3j-12i-23)\right\}_{g-1} & \mbox{if $j\equiv 1\pmod{2}$,}\\
\left\{g+3i+5,\frac{1}{2}(2g-3j+6i+10)\frac{1}{2}(2g+3j-12i-20)\right\}_g & \mbox{if $j\equiv 0\pmod{2}$.} \\
\end{cases}
\]
For each $i\in\{0,1,\ldots,(n-8)/3\}$ and each $j\in\{1,2,\ldots,3i+5\}$, when $n\equiv 5\pmod{6}$, let
\[
U_{i,j}=\begin{cases}
\left\{\frac{1}{2}(2g-3j+12i+25),\frac{1}{2}(2g+3j-6i-11),g-3i-7\right\}_{g} & \mbox{if $j\equiv 1\pmod{2}$,}\\ 
\left\{\frac{1}{2}(2g-3j+12i+22),\frac{1}{2}(2g+3j-6i-14),g-3k-7\right\}_{g-1} & \mbox{if $j\equiv 0\pmod{2}$,}\\
\end{cases}
\]
Let $U_{0,5}=\{g+2,g-1,g-4\}_{g-1}$ for $n\equiv 1\pmod{6}$. Let $U_{0,6}=\{g+5,g-1,g-4\}_{g}$ and $U_{0,7}=\{g+2,g-1,g-4\}_{g_1}$ for $n\equiv 5\pmod{6}$. 
For each $i,j$, let $\overline{U}_{i,j}=\{a',b',c'\}$ be a vertex where $U_{i,j}=\{a,b,c\}$ and $a'=2g-1-a$, $b'=2g-1-b$, $c'=2g-1-c$. 
Two vertices $U_{i_1,j_1}$ and $U_{i_2,j_2}$ are adjacent if $|U_{i_1,j_1}\cap U_{i_2,j_2}|=2$. It is clear that each element in the triples $U_{i,j}$ and $\overline{U}_{i,j}$ is an integer because of $j$'s parity.
Because the sum of the elements in each triple is either $3g-3$ or $3g$ modulo $n$, $U_{i,j}$ and $\overline{U}_{i,j}$ are indeed triples from Step $1a$. 
These adjacencies are illustrated in Figure~\ref{hex19}, which gives $H^{19}_{g-1,g}$ with the embedded $H^{13}_{g-1,g}$ in bold on the left.  On the right is $H^{17}_{g-1,g}$ with the embedded $H^{11}_{g-1,g}$ in bold. Figure~\ref{hex19} shows that $H^{n-6}_{g-1,g}$ is in fact a subgraph of $H^n_{g-1,g}$. 
In the appendix, we give a formal proof that the above description produces the honeycomb graph $H^{n-6}_{g-1,g}$ as a subgraph of $H^n_{g-1,g}$. 

\begin{figure}[htb]
\begin{center}
 \includegraphics[scale=.8]{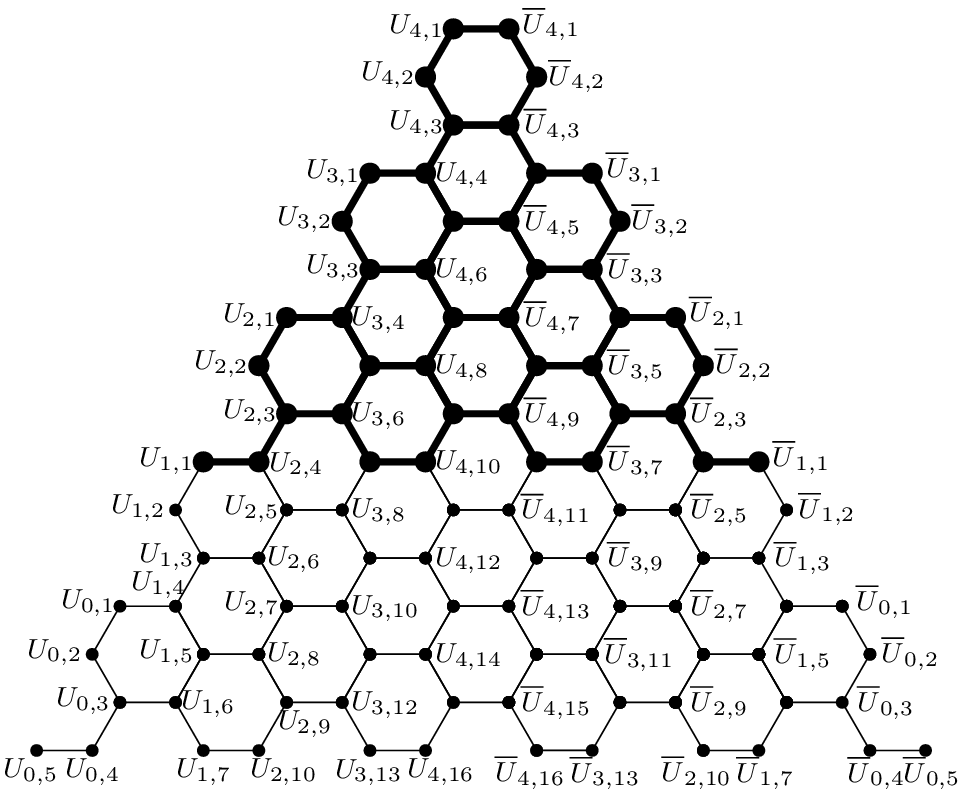}
 \hspace{.2in}
 \includegraphics[scale=.8]{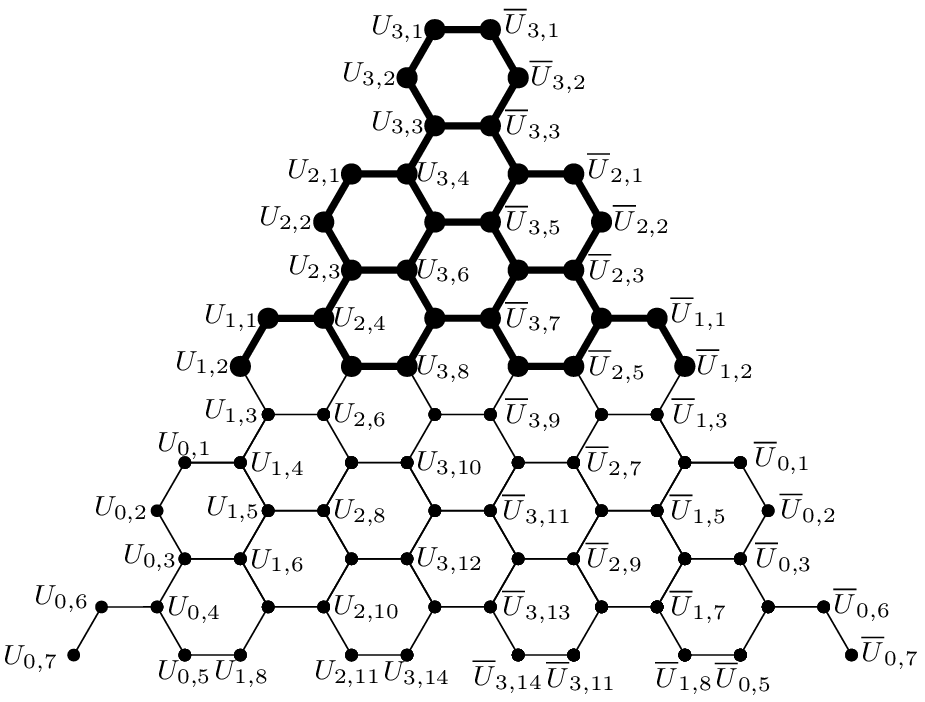}
\end{center}
\caption{Honeycomb graphs $H^{19}_{g-1,g}$ (left) and $H^{17}_{g-1,g}$ (right) with the embedded honeycomb graphs $H^{13}_{g-1,g}$ and $H^{11}_{g-1,g}$ respectively}\label{hex19}
\end{figure}

 Let $W^{n-6}_{g-1,g}$ be the Hamilton walk in $H^{n-6}_{g-1,g}$ from the induction hypothesis. 
 We will now form a Hamilton walk in $H^n_{g-1,g}$, by concatenating 3-paths and $W^{n-6}_{g-1,g}$.
If $n\equiv 1\pmod{6}$, for $k\in \{1,\ldots,\frac{2n-17}{3}\}$, we define the $k^{th}$ 3-path to be
\[
P_{k} = \begin{cases}
[U_{k,3k-1},U_{k,3k},U_{k,3k+1},U_{k,3k+2}] &\mbox{if $k$ is odd,} \\
[U_{k,3k+2},U_{k,3k+1},U_{k,3k},U_{k,3k-1}] &\mbox{if $k$ is even.} \\
\end{cases}
\]
If $n\equiv 5\pmod{6}$, for $k\in\{1,\ldots,\frac{2n-25}{3}\}$ we define the $k^{th}$ 3-path to be
\[
P_{k} = \begin{cases}
[U_{k,3k+3},U_{k,3k+2},U_{k,3k+1},U_{k,3k}] &\mbox{if $k$ is odd,} \\
[U_{k,3k},U_{k,3k+1},U_{k,3k+2},U_{k,3k+3}] &\mbox{if $k$ is even.} \\
\end{cases}
\]

\begin{figure}[htb]
\begin{center}
  \includegraphics[scale=.8]{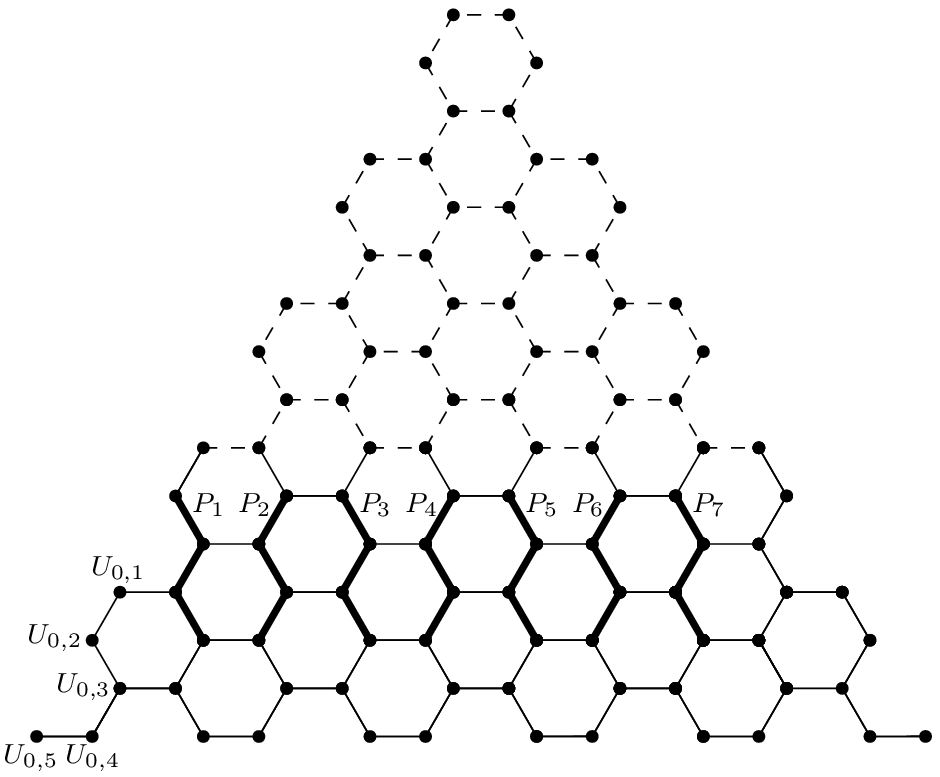}
 \hspace{.2in}
 \includegraphics[scale=.8]{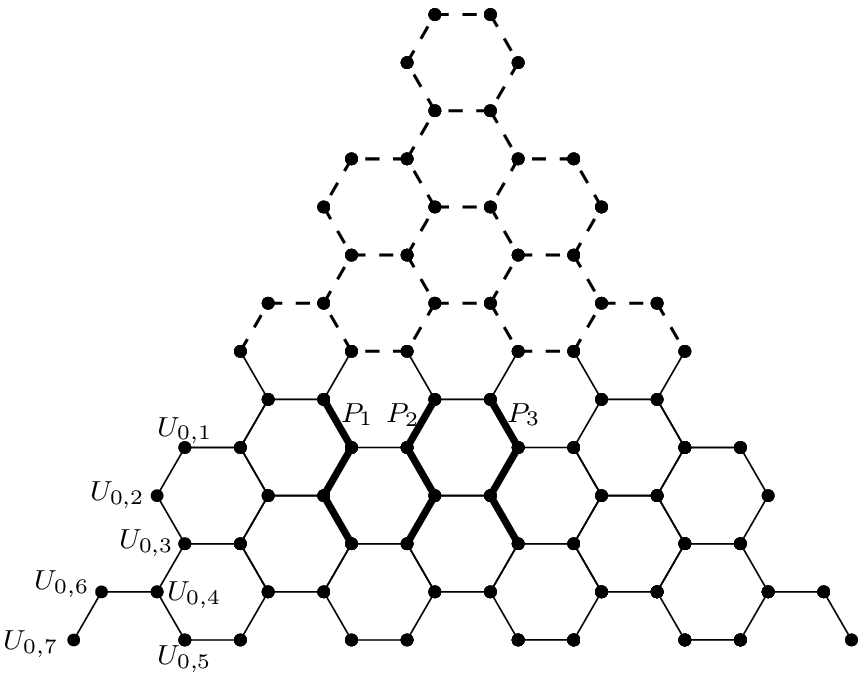}
\end{center}
\caption{Paths $P_i$ are highlighted in $n\equiv 1\pmod{6}$ case ($n=19$ on left) and in $n\equiv 5\pmod{6}$ case ($n=17$ on right)}\label{hex19pathP}
\end{figure}

The paths $P_k$ are highlighted in Figure~\ref{hex19pathP}.
In a similar manner, we have the following $3$-paths for $k'\in\{1,\ldots,\frac{n-7}{3}\}$ when $n\equiv 1\pmod{6}$ and for $k'\in\{1,\ldots,\frac{n-11}{3}\}$ when $n\equiv 5\pmod{6}$:
\[
T_{k'}=\begin{cases}
[U_{2k'-2,6k'-3},U_{2k'-1,6k'},U_{2k'-1,6k'+1},U_{2k',6k'+4}] &\mbox{if $n\equiv 1\pmod{6}$,} \\
[U_{2k'-1,6k'+1},U_{2k',6k'+4},U_{2k',6k'+5},U_{2k'+1,6k'+8}] &\mbox{if $n\equiv 5\pmod{6}$.} \\
\end{cases}
\]

\begin{figure}[htb]
\begin{center}
  \includegraphics[scale=.8]{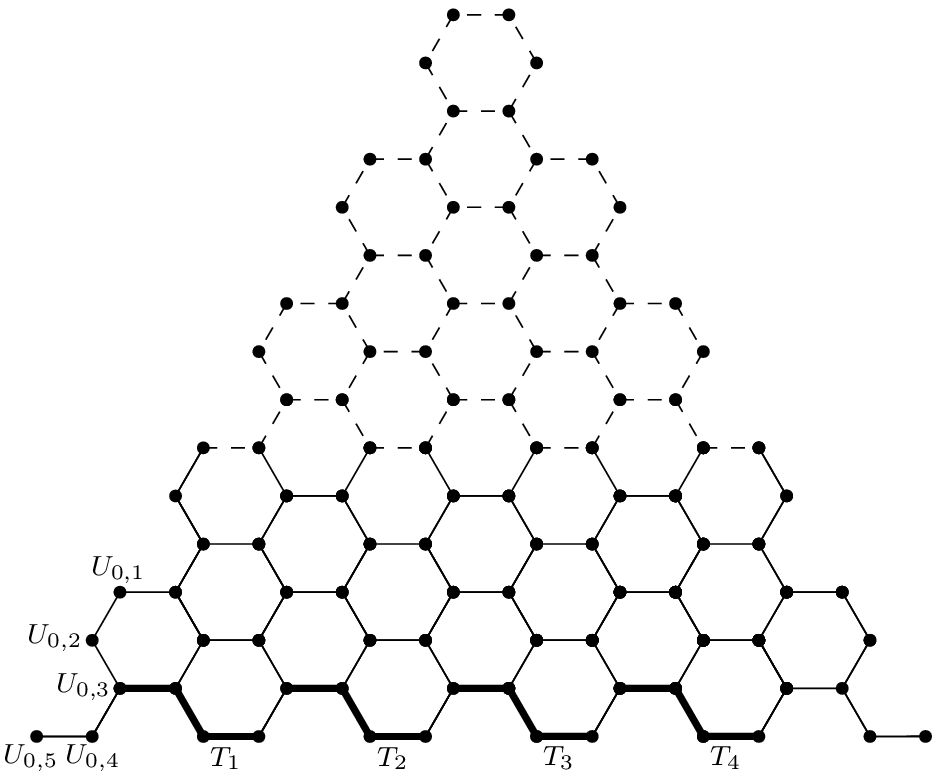}
 \hspace{.2in}
 \includegraphics[scale=.8]{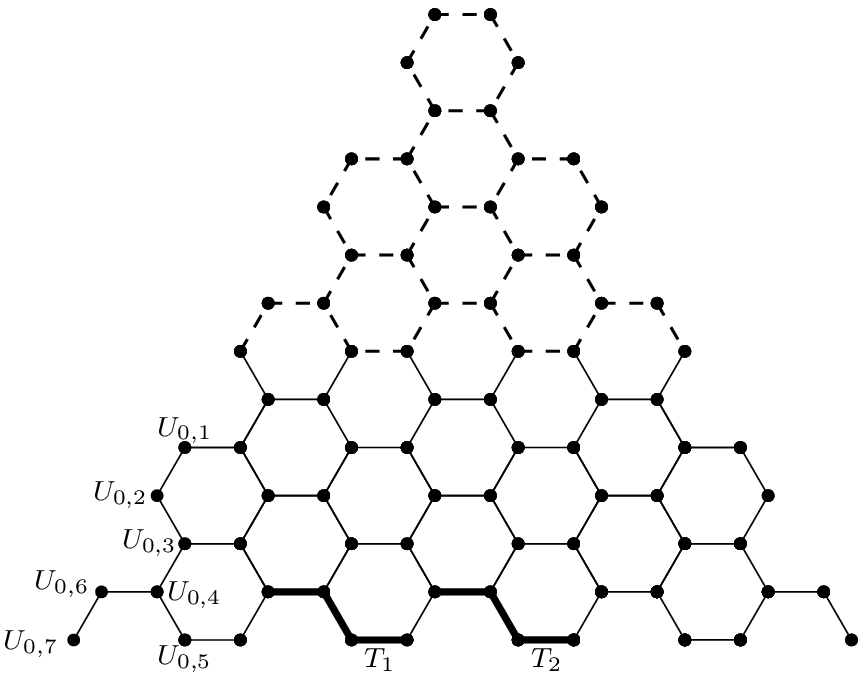}
\end{center}
\caption{Paths $T_i$ are highlighted in $n\equiv 1\pmod{6}$ case ($n=19$ on left) and in $n\equiv 5\pmod{6}$ case ($n=17$ on right)}\label{hex19pathQ}
\end{figure}

These paths are highlighted in Figure~\ref{hex19pathQ}.
Then we can define $W^{n}_{g-1,g}$ to be the following Hamilton walk. (The parts of this walk not including $W^{n-6}_{g-1,g}$ are highlighted in Figure~\ref{hex19fullPath}.)

\[
W^{n}_{g-1,g}=\begin{cases}
[U_{0,5},U_{0,4},U_{0,3},U_{0,2},U_{0,1},U_{0,2}]\circ T_1 \circ T_2\circ \cdots \circ T_{(n-7)/3} \circ [\overline{U}_{1,6},\overline{U}_{1,5}] \circ  & \\
P_{(2n-17)/3} \circ \cdots \circ P_2 \circ P_1 \circ W^{n-6}_{g-1,g} \circ[\overline{U}_{1,2},\overline{U}_{1,3},\overline{U}_{1,4},\overline{U}_{0,1},\overline{U}_{0,2},\overline{U}_{0,3},\overline{U}_{0,4},\overline{U}_{0,5}] &\mbox{if $n\equiv 1\pmod{6}$,} \\
& \\
[U_{0,7},U_{0,6},U_{0,4},U_{0,5},U_{1,8},U_{0,5},U_{0,4},U_{0,3},U_{0,2},U_{0,1},U_{1,4},U_{1,5},U_{1,6}]\circ T_1\circ & \\
 T_2 \circ \cdots \circ T_{(n-11)/3} \circ[\overline{U}_{2,10},\overline{U}_{2,9}] \circ P_{(2n-25)/3}\circ \cdots \circ P_2\circ P_1 \circ [U_{1,3}]\circ W^{n-6}_{g-1,g}\circ &\mbox{if $n\equiv 5\pmod{6}$.} \\
[\overline{U}_{1,3},\overline{U}_{2,6},\overline{U}_{2,7}, \overline{U}_{2,8},\overline{U}_{1,5},\overline{U}_{1,4}, \overline{U}_{0,1},\overline{U}_{0,2},\overline{U}_{0,3},\overline{U}_{1,6},\overline{U}_{1,7},\overline{U}_{1,8},\overline{U}_{0,5},\overline{U}_{0,4},& \\
\overline{U}_{0,6},\overline{U}_{0,7}] & \\
\end{cases}
\]
\begin{figure}[htb]
\begin{center}
  \includegraphics[scale=.8]{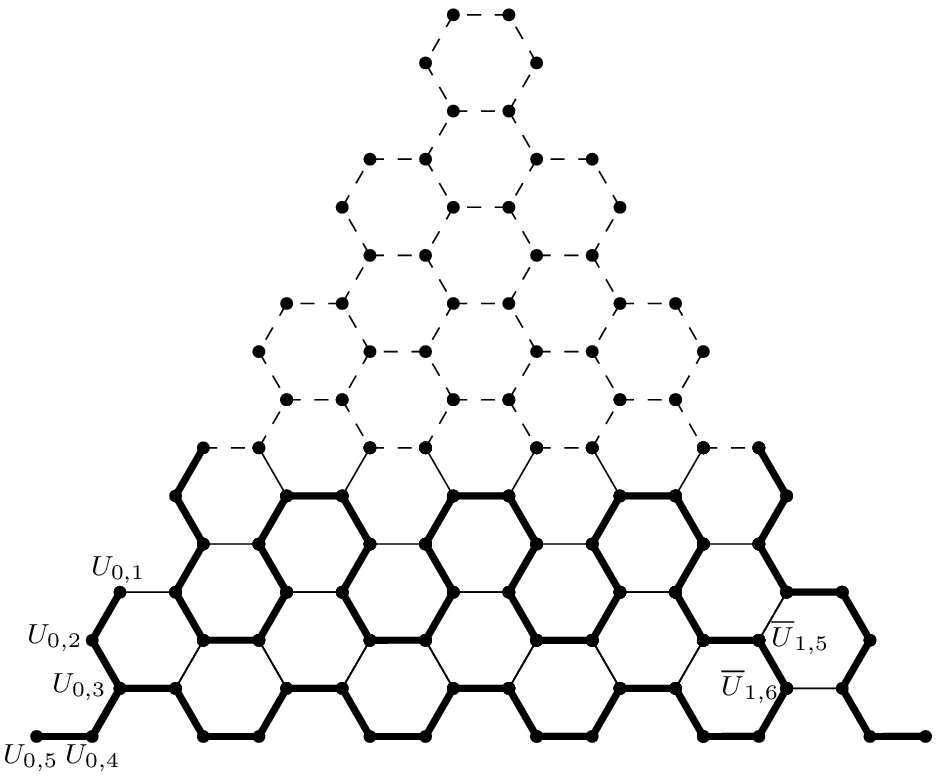}
  \hspace{.2in}
  \includegraphics[scale=.8]{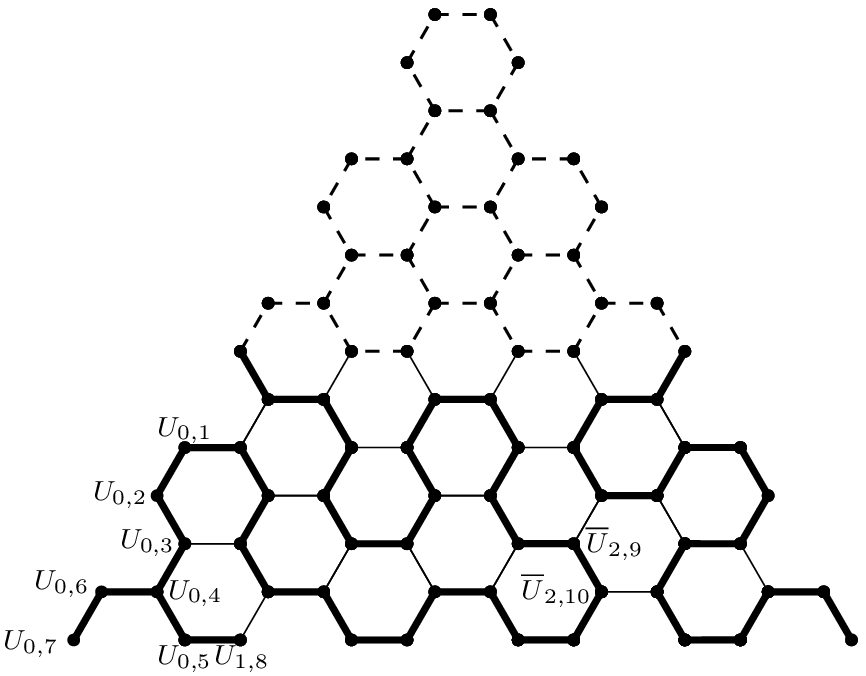}
\end{center}
\caption{Walk taken during induction step for $n\equiv 1\pmod{6}$ ($n=19$ on left) and $n\equiv 5\pmod{6}$ ($n=17$ on right)}\label{hex19fullPath}
\end{figure}

\noindent Thus $W^{n}_{g-1,g}$ is a $(Z_1,Z_2)$-Hamilton walk in $H^n_{g-1,g}$.

Finally, we must show that there is a $(Z_1,Z_2)$-Hamilton walk through $H^n_{\frac{\l}{2}-3,\frac{\l}{2}-2,\frac{\l}{2}-1}$. Let $g=\frac{\l}{2}-2$; then there is a Hamilton walk through $H^{n}_{g-1,g}$. We will show there is also a Hamilton walk through $H^{n}_{g-1,g,g+1}$.
Suppose $\{\alpha,\beta,\gamma\}_{g+1} \in \s_{g+1}$ is a triple obtained by Step $1a$. Then $\alpha+\beta+\gamma=3(g+1)$. But then $\alpha+\beta+(\gamma-3)=3g$ and so $\{\alpha,\beta,\gamma-3\}_{g} \in \s_{g}$. Furthermore, these two triples correspond to adjacent vertices in $H^n_{g-1,g,g+1}$. Thus any vertex $v \in V(H^n_{g-1,g,g+1})$ corresponding to a triple from $\s_{g}$ is adjacent to some vertex $v' \in V(H^n_{g-1,g,g+1})$ corresponding to a triple from $\s_{g+1}$. 
Suppose the  $(Z_1,Z_2)$-Hamilton walk through $H^n_{g-1,g}$ is  $(Z_1,\ldots,v,\ldots,Z_2)$ where $v$ is the vertex corresponding to $\{\alpha,\beta,\gamma-3\}_{g}$. 
Form a new walk which also contains $v'$: $(Z_1,\ldots, v, v',v, \ldots,Z_2)$.
We may follow the same process for each of the vertices in $V(H^n_{g-1,g,g+1})$ that correspond to triples from $\s_{g+1}$, thus creating a  $(Z_1,Z_2)$-Hamilton walk through $H^n_{g-1,g,g+1}$.

\end{proof}

Because $Z$ is an endpoint of the Hamilton walk given in Lemma~\ref{2hexagons}, removing $Z$ from the walk induces a Hamilton walk on the remaining graph. Let $H=H^n_{0,1, \ldots,\frac{\l}{2}-1} \backslash Z$; we have the following corollary.

\begin{corollary}
\label{HamWalk}
$H$ contains a Hamilton walk.
\end{corollary}

The subgraph of the $2$-BIG formed by taking one triple from Step $1a$ and its corresponding 7 triples from Step $1b$ of Construction~\ref{sblocks} is isomorphic to the 3-cube, $Q_{3}$ as shown in Figure~\ref{pointExplosion}.

\begin{figure}[htb]
\begin{center}
	\includegraphics[scale=1]{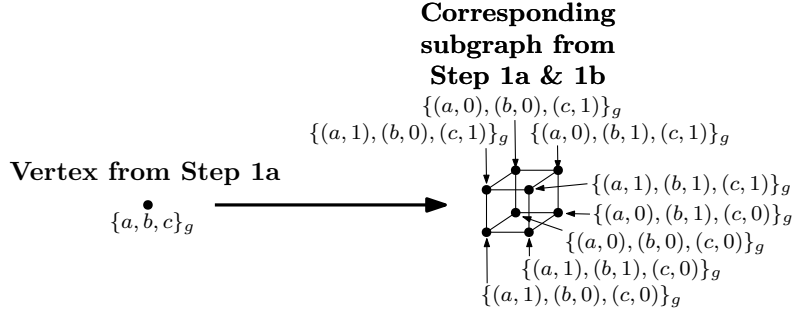}
\end{center}
\caption{Correspondence between a single vertex in Step $1a$ and 8 vertices in Step $1b$}\label{pointExplosion}
\end{figure}

The {\em cartesian product} $G_1 \mathbin{\square} G_2$ of graphs $G_1$ and $G_2$ is a graph whose vertex set is the 
cartesian product $V(G_1) \times V(G_2)$; and any two vertices $(u,u')$ and $(v,v')$ are adjacent in $G_1 \mathbin{\square} G_2$ if and only if either
\begin{itemize}
\item $u=v$ and $u'$ is adjacent with $v'$ in $G_2$, or
\item $u'=v'$ and $u$ is adjacent with $v$ in $G_1$.
\end{itemize}
Thus the subgraph of the $2$-BIG corresponding to Steps $1a$ and $1b$ from Construction~\ref{sblocks} is the cartesian product of 
$H^n_{0,1, \ldots,\frac{\l}{2}-1}$ and $Q_{3}$. Batagelj and Pisanski \cite{BP} have shown the following.
	
\begin{theorem}
\label{product}
{\normalfont\cite{BP}} Suppose $T$ is a tree and $M$ is a Hamiltonian graph with $|V(M)|=n$. If $\Delta(T) < n$, then $T \square M$ is Hamiltonian.
\end{theorem}

Let $W$ be the subgraph of $H$ that is induced by the vertices of the Hamilton walk obtained by Corollary~\ref{HamWalk}. Then it is clear that $W$ is
a tree and $\Delta(W)=3$. Thus since $Q_3$ is Hamiltonian, we have the following result.

\begin{lemma}
\label{HamProduct}
$W \square Q_{3}$ is Hamiltonian.
\end{lemma}

Because $W$ spans the vertices of $H$, we have a Hamilton cycle through the subgraph of the 2-BIG that is formed from the blocks of the $\TS(v,\l)$ that are given in Steps $1a$ and $1b$ of Construction~\ref{sblocks}, except the block $Z=\{n-2,1,4\}$.

\section{Subgraph Corresponding to Steps $1c$, $2a$, $2b$, and $3$}\label{path_2}

In this section we show that there exists a Hamilton path through the portion of the 2-BIG of $\ts(v,\lambda)$ that is formed from the blocks of the $\TS(v,\lambda)$ that are given in Steps $1c$, $2a$, $2b$, and 3 of Construction~\ref{sblocks}.

It is well known that any integer can be represented using negative bases. 
Negative bases were first studied by Vittorio Gr\"{u}nwald in 1885. For the purposes of our construction, we only care about one particular negative base.

\begin{theorem}\label{integersum}
Let $s$ be any integer. Then $s$ can be written as
\[
s=\sum_{k=1}^\infty (-2)^{k-1}y_k
\]
where $y_k\in \{0,1\}$ for each $k\geq 1$. 
\end{theorem}

The {\em underlying graph} $H$ of a directed graph $H'$ is a graph that has the same vertex set and edge set as $H'$, but all of the edges in $H$ have no direction. A directed path in $H'$ is a path with the property that for each pair of consecutive arcs $e_1$ and $e_2$,  the head of $e_1$ is adjacent to the tail of $e_2$. We define $G\cup H$ as the union of $G$ and $H$, not necessarily vertex disjoint. We say that a directed graph is \textit{strongly connected} if there exists a directed path from any vertex to any other vertex.

\begin{lemma}\label{digraphConnected}
Let $g\in \Z_n$ for $n\equiv 1$ or $5\pmod{6}$. Let $D_g'$ be a directed graph with vertex set $\{0,\ldots,n-1\}$ such that there is a directed arc from $i+g$ to $-2i+g$  
for each $i\in \Z_n$ where calculations are done modulo $n$. Then $D_g' \cup D_{g+1}'$ is strongly connected.
\end{lemma}

\begin{proof}
Let $x$ be a vertex. Since all calculations will be identical for $D_g'$ as they are for $D_0'$, without loss of generality, let $g=0$. We define a mapping $\alpha: \Z_n\to \Z_n$ by $\alpha(x) = -2x$. 
Notice that this mapping represents the arcs in $D_0'$.
That is, in $D_0'$, there is a directed edge from $x$ to $-2x$ for each $x\in \Z_n$. 
Therefore, in $D_1'$, we can define the mapping $\beta: \Z_n\to \Z_n$ to represent the arcs in $D_1'$ by $\beta(x)= -2(x-1)+1=-2x+3$. Thus we must show that for any pair of vertices $\{x,y\}$, there is a directed path from $x$ to $y$. Equivalently, we must show there is a composition of some number of mappings $\alpha$ and $\beta$ that will send $x$ to $y$.

Notice that if we can show that any $x$ can be mapped to $x+1$ by a sequence of compositions of $\alpha$ and $\beta$, we can map $x$ to any $z$ by mapping $x$ to $x+1$, mapping $x+1$ to $x+2$, mapping $x+2$ to $x+3$, and so on. 
In this manner, we will form a directed walk. But since we could form a directed path from $x$ to $z$ from this directed walk, it is enough to show that there is a directed walk from $x$ to $x+1$ for any $x$. 
Without loss of generality, we may assume $x=1$. 
Thus we aim to show that we can map $1$ to $2$.
Represent the mapping from $x$ to $z$ as a composition of $t$ mappings, where each mapping is either $\alpha$ or $\beta$. So we define $z$ as a function of $x$:
\begin{align*}
z(x) &= -2(-2(\cdots(-2(-2(-2x+3y_{t-1})+3y_{t-2})+3y_{t-3})\cdots)+3y_1)+3y_0 \\
&= (-2)^t+3(-2)^{t-1}y_{t-1}+3(-2)^{t-2}y_{t-2}+\cdots + 3(-2)y_1+3y_0 \\
&= (-2)^tx-6\left(\sum_{k=1}^{t-1} (-2)^{k-1}y_k\right) + 3y_0
\end{align*}
where $y_0,y_1,\ldots,y_{t-1}$ are either $0$ or $1$, and
\[
y_i=
\begin{cases}
0 &\mbox{if $\alpha$ is applied, or} \\
1 & \mbox{if $\beta$ is applied.}
\end{cases}
\]
Because we assume $x=1$, we may choose $y_0=0$ and thus we map $1$ to $z$ by
\[
z(x)\equiv (-2)^t +(-6)\left(\sum_{k=1}^{t-1} (-2)^{k-1} y_k\right)\pmod{n}
\]
where $y_1,y_2,\ldots,y_{t-1}\in \{0,1\}$, $y_0=0$. Let $t$ be large enough to be a multiple of order of $-2$ in $\Z_n$; denote this order as $a$. We can say this without loss of generality since writing a base-$\ell$ number as $100$ is the same as writing this base-$\ell$ number as $000100$. 
Since $t$ is a multiple of $a$, we have:
\[
z(x)\equiv 1+(-6)\left(\sum_{k=1}^{t-1} (-2)^{k-1} y_k\right)\pmod{n}.
\]
Since $n\equiv 1\pmod{6}$, we can write $n=6s+1$ for some $s$, so $s=\frac{n-1}{6}$. By Lemma~\ref{integersum}, choose $y_1,y_2,\ldots,y_{t-1}$ so that 
\[
s=\sum_{k=1}^{t-1}(-2)^{k-1}y_{k-1}.
\]
Then we have
\[
z(x) \equiv 1 +(-6)s \\
\equiv 1+(-6)\left(\frac{n-1}{6}\right) \\
\equiv 2\pmod{n}.
\]
Since we only used mappings $\alpha$ and $\beta$, and not their inverses, the walk described using the mappings $\alpha$ and $\beta$ is directed and so the directed graph $D_g'\cup D_{g+1}'$ is strongly connected. A similar argument can be made when $n\equiv 5\pmod{6}$.
\end{proof}

Let $g \in \Z_n$ for $n\equiv 1$ or $5 \pmod{6}$.  Let $D'_{g}$ be the union of all directed graphs given in Construction~{\normalfont\ref{sblocks}} and $G_g$ be the induced subgraph of the $2$-BIG of the $\ts(v,2)$ formed by the blocks of $\s_g$ given in Steps $1c$, $2a$, $2b$, and $3$. By Construction~{\normalfont\ref{sblocks}}, each arc in $D_g'\cup D_{g+1}'$ corresponds to a subgraph of $G_g\cup G_{g+1}$ on $6$ vertices that is isomorphic to one of the graphs in Figure~\ref{doubleStar}, depending on the color of the arc applied to $D_g'\cup D_{g+1}'$ in Step $0$.
Note that the vertices of degree $1$ in either graph in Figure~\ref{doubleStar} are from either Step $2a$ or Step $2b$. Careful inspection of the construction shows that if two adjacent arcs $(a,b)$ and $(b,c)$ are in $D'_g\cup D'_{g+1}$, then the subgraphs of $G_g \cup G_{g+1}$ which correspond to these arcs are connected and form the subgraph in Figure~\ref{connectedDoubleStar} up to isomorphism, regardless of the color of the two arcs chosen in Step 0. Note that since we deal exclusively with directed trails, we will never have a pair of consecutive directed arcs in our directed trails of the form $\{(a,b),(c,b)\}$ or $\{(a,b),(a,c)\}$. Because there is a directed arc from $i+g$ to $-2i+g$ for each $i\in \Z_n$, we get that $D_g'\cup D_{g+1}'$ is strongly connected by Lemma~\ref{digraphConnected}.  Because $D_g'\cup D_{g+1}'$ is connected, it follows that $G_g \cup G_{g+1}$ is connected as well. Thus we have the following result.

\begin{lemma}\label{sg+sg+1connected}
Let $g\in \Z_n$ for $n\equiv 1$ or $5 \pmod{6}$. Let $G_g \cup G_{g+1}$ be the subgraph of the $2$-BIG of the $\ts(v,4)$ formed by the blocks of $\s_g$  and $\s_{g+1}$ given in Steps $1c$, $2a$, $2b$, and $3$ of Construction~{\normalfont\ref{sblocks}}.  Then $G_g\cup G_{g+1}$ is connected.
\end{lemma}

\begin{figure}[htb]
\begin{center}
 \includegraphics[scale=1]{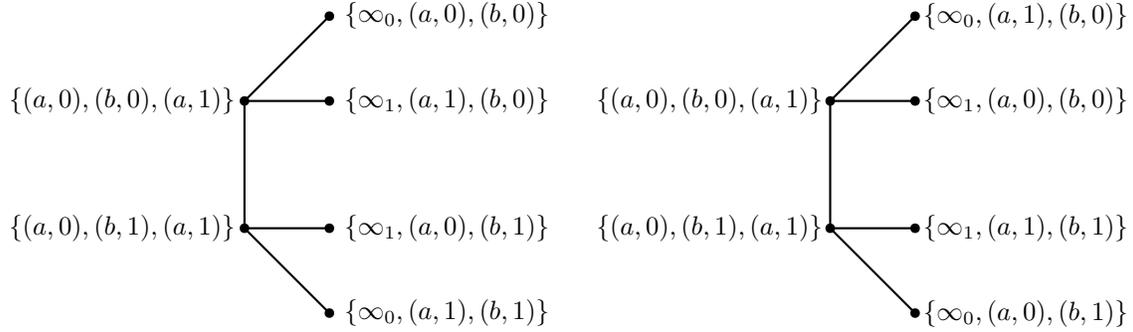}
\end{center}
\caption{Subgraph of $\ts(v,4)$ formed from either a single red arc (left) or blue arc (right) of $D_g$}\label{doubleStar}
\end{figure}

\begin{figure}[htb]
\begin{center}
 \includegraphics[scale=1]{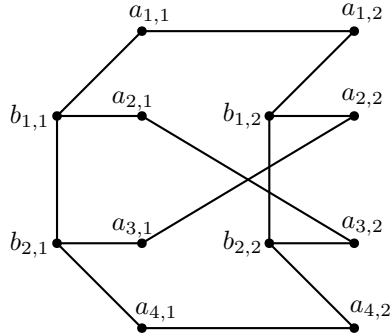}
\end{center}
\caption{Subgraph of 2-BIG of $\ts(v,4)$ formed from two adjacent arcs in $D_g$}\label{connectedDoubleStar}
\end{figure}

To build a Hamilton path in the $2$-BIG of $\ts(v,\l)$ formed by the blocks from Steps $1c$, $2a$, $2b$, and $3$, we begin by finding a Hamilton path through two consecutive arcs on a directed trail in the subgraph of the $2$-BIG of $\ts(v,\l)$ formed by the blocks in Steps $1c$, $2a$, and $2b$.

\begin{lemma}\label{etrail_ham}
Let $v=2n+2$ and $n\equiv 1$ or $5\pmod{6}$. Let $D'=\bigcup_{g\in\Z_{\l/2}} D_g'$ be the union of all directed graphs given in Construction~{\normalfont\ref{sblocks}}, and let $G$ be the subgraph of the $2$-BIG of $\TS(v,\l)$ formed by the blocks given in Steps $1c$, $2a$, and $2b$. Let $H$, a subgraph of $G$, be the graph formed by selecting two arcs from $D'$ that form a directed path $[e_1,e_2]$ in $D'$. 
There is a Hamilton path in $H$ that begins at a vertex in $H$ corresponding to $e_1$ in Step $2a$ or $2b$ and ends at a vertex in $H$ corresponding to $e_2$ in Step $2a$ or $2b$. 
\end{lemma}

\begin{proof}
Because the two arcs $e_1,e_2$ form a directed trail in $D'$, $H$ is isomorphic to the graph in Figure~\ref{connectedDoubleStar}.
A Hamilton path of the prescribed type is 
\[a_{1,1},b_{1,1}, a_{2,1} ,a_{3,2} ,b_{2,2} ,a_{4,2} ,a_{4,1} ,b_{2,1} ,a_{3,1} ,a_{2,2} ,b_{1,2} ,a_{1,2}.\]
\end{proof}

Let $D'=\bigcup_{g\in\Z_{\l/2}} D_g$ be the union of all directed graphs formed from Step 0 of Construction~\ref{sblocks}, and let $H$ be the induced subgraph of the $2$-BIG of a $\TS(v,\lambda)$ formed by the blocks given in Steps $1c$, $2a$, and $2b$. The following lemma shows that we can find a Hamilton path in the subgraph
of $H$ induced by a directed trail in $D'$ with both endpoints corresponding to the same arc in the trail.  

\begin{lemma}\label{part_etrail_ham}
Let $v=2n+2$ and $n\equiv 1$ or $5\pmod{6}$. Let $D'=\bigcup_{g\in\Z_{\l/2}}D_g$ be the directed graph given by Step 0 in Construction~\ref{sblocks}, and let $H$ be the subgraph of the $2$-BIG of $\TS(v,\l)$ formed by the blocks given in Steps $1c$, $2a$, and $2b$. Let $\widehat{H}$, a subgraph of $H$, be formed by selecting a directed trail of length $k \geq 3$, $T=[e_0,e_1,\ldots,e_{k-1}]$, from $D'$. Then there is a Hamilton path $P$ in $\widehat{H}$ that begins at a vertex in $\widehat{H}$ corresponding to $e_0$ (or $e_{k-1}$) in Step $2a$ or Step $2b$ and ends at a vertex in $\widehat{H}$ corresponding to $e_0$ (or $e_{k-1}$) in Step $2a$ or Step $2b$. 
\end{lemma}

\begin{proof}
Though each pair of consecutive directed arcs in $T$ forms a graph isomorphic to the graph in Figure~\ref{connectedDoubleStar}, there are many ways that three directed arcs can form a subgraph in $H$. Label the vertices in $T$ formed by $e_0$ and $e_1$ as in Figure~\ref{connectedDoubleStar} and suppose without loss of generality that the Hamilton path $P$ we wish to construct begins at $a_1$. 
Then the first six vertices of $P$ will be $a_{1,1},b_{1,1},a_{2,1},a_{3,2},b_{2,2},a_{4,2}$. 
Depending on the colors of $e_0$, $e_1$, and $e_2$, the subgraph in $H$ formed by the $18$ vertices represented by $e_0$, $e_1$, $e_2$ is one of the two graphs in Figure~\ref{18verts_directed}.
A simple permutation of the right-most six vertices in the graph on the right in Figure~\ref{18verts_directed} shows that the 
two graphs in Figure~\ref{18verts_directed} are isomorphic, but this permutation may change the way the 12 vertices are connected between $e_2$ and $e_3$. 
However, a similar permutation of the six vertices represented by $e_3$ can be applied so that the subgraph formed by the 18 vertices represented by $e_0,e_1,e_2$ is isomorphic to the subgraph formed by the $18$ vertices represented by $e_1,e_2,e_3$. 
So we may assume that any set of 18 vertices formed from $e_i,e_{i+1},e_{i+2}$ is isomorphic to the graph on the left in Figure~\ref{18verts_directed}.
We can form a path, $P_1$, through half of the vertices represented by $e_0,e_1,\ldots,e_{k-3}$ as shown by the solid black path in Figure~\ref{18verts_directed_ham3}.
Depending on the value of $k\pmod{4}$, there are four possibilities for a path,  $P_3$, through the 12 vertices represented by $e_{k-2}$ and $e_{k-1}$ as shown in Figure~\ref{18verts_directed_ham}. 
Whichever path is chosen for $P_3$, join $P_3$ to the dotted path $P_2$ through $e_{k-3},e_{k-4},\ldots,e_{0}$ shown in Figure~\ref{18verts_directed_ham3}. The dashed edges represent edges that are not used in any path. 
Thus there is a Hamilton path $P_1 \circ P_3 \circ P_2$ in $\widehat{H}$ using all arcs in the directed trail $T$ that begins at vertex $a_1$ and ends at vertex $a_4$ in $\widehat{H}$ (see Figure~\ref{18verts_directed_ham3}). 

Now we give the explicit construction of this path.
For $i=0,1, \ldots, k-1$, label the vertices in $\widehat{H}$ that correspond to $e_{i}$ with the elements in the set $\{a_{1,i},a_{2,i},a_{3,i},a_{4,i},b_{1,i},b_{2,i}\}$ as in Figure~\ref{connectedDoubleStar}.
The Hamilton path through $H$ will consist of three subpaths $P_1$, $P_2$ and $P_3$, which we now describe. 
Define the 3-paths $p_{i}$ and $\overline{p}_{i}$ for any $e_i$ depending on the value of $i$ modulo 4 as follows:

\[
p_i=\begin{cases}
[a_{1,i},b_{1,i},a_{2,i}] &\mbox{if $i \equiv 0 \pmod{4}$} \\
[a_{3,i},b_{2,i},a_{4,i}] &\mbox{if $i\equiv 1\pmod{4}$} \\
[a_{4,i},b_{2,i},a_{3,i}] &\mbox{if $i\equiv 2\pmod{4}$} \\
[a_{2,i},b_{1,i},a_{1,i}] &\mbox{if $i\equiv 3\pmod{4}$}\\
\end{cases}
\]
and
\[
\overline{p}_i=\begin{cases}
p_{i-1} &\mbox{if $i \equiv 1,3 \pmod{4}$} \\
p_{i+1} &\mbox{if $i\equiv 0,2\pmod{4}$.} \\
\end{cases}
\]

Then $P_{1}= p_0 \circ p_1 \circ \ldots p_{k-3}$ and $P_2=\overline{p}_{k-3} \circ \overline{p}_{k-4} \ldots \circ \overline{p}_{0}$; these are illustrated in Figure~\ref{18verts_directed_ham3} when $k=5$.  We define $P_{3}$ based on one of the endpoints of $P_{1}$. This is illustrated in Figure~\ref{18verts_directed_ham}
and is defined as follows.

\[
P_3=\begin{cases}
a_{4,k-3},  a_{4,k-2}, a_{4,k-1}, b_{2,k-1}, a_{3,k-1}, a_{2,k-2}, b_{1,k-2},&\\
b_{2,k-2}, a_{3,k-2}, a_{2,k-1}, b_{1,k-1}, a_{4,k-1}, a_{4,k-2}, a_{1,k-3} &\mbox{if $k \equiv 0 \pmod{4}$,} \\
&\\
a_{3,k-3},  a_{2,k-2}, a_{3,k-1}, b_{2,k-1}, a_{4,k-1}, a_{4,k-2}, b_{2,k-2},&\\
b_{1,k-2}, a_{1,k-2}, a_{1,k-1}, b_{1,k-1}, a_{2,k-1}, a_{3,k-2}, a_{2,k-3} &\mbox{if $k \equiv 1 \pmod{4}$,} \\
&\\
a_{1,k-3},  a_{1,k-2}, a_{2,k-1}, b_{1,k-1}, a_{2,k-1}, a_{3,k-2}, b_{2,k-2},&\\
b_{1,k-2}, a_{2,k-2}, a_{3,k-1}, b_{2,k-1}, a_{4,k-1}, a_{4,k-2}, a_{4,k-3} &\mbox{if $k \equiv 2 \pmod{4}$, and} \\
&\\
a_{2,k-3},  a_{3,k-2}, a_{2,k-1}, b_{1,k-1}, a_{1,k-1}, a_{1,k-2}, b_{1,k-2},&\\
b_{2,k-2}, a_{4,k-2}, a_{4,k-1}, b_{2,k-1}, a_{3,k-1}, a_{2,k-2}, a_{3,k-3} &\mbox{if $k \equiv 3 \pmod{4}$,} \\

\end{cases}
\]

Then $P=P_1 \circ P_3 \circ P_2$ is the desired Hamilton path.
\end{proof}

\begin{figure}[htb]
\begin{center}
\includegraphics[scale=1]{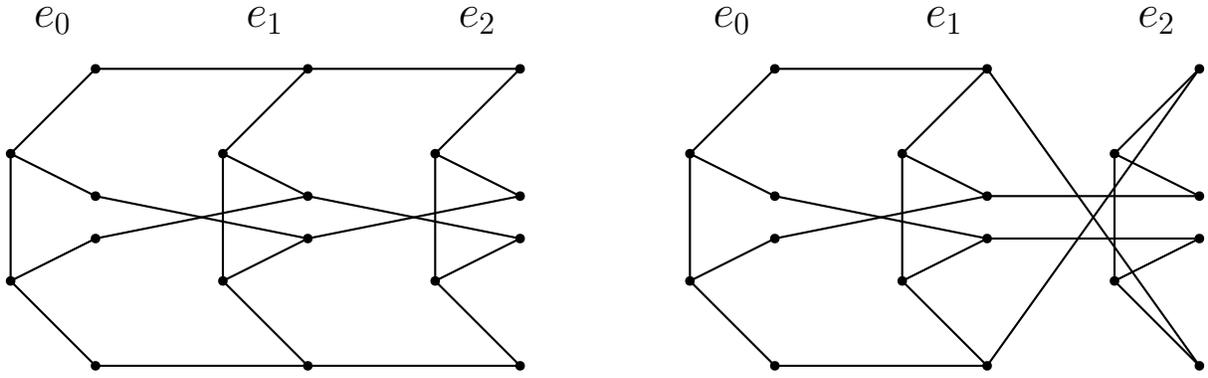}
\end{center}
\caption{Possibilities for the first three arcs of $T$}\label{18verts_directed}
\end{figure}

\begin{figure}[htb]
\begin{center}
	\includegraphics[scale=1]{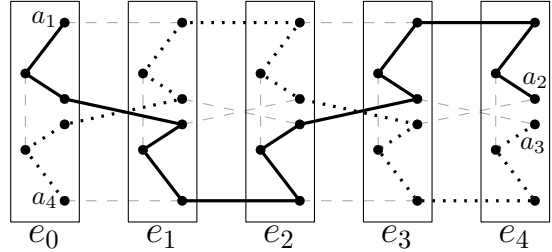}
\end{center}
\caption{The paths $P_{1}$ (solid) and $P_{2}$ (dotted)}\label{18verts_directed_ham3}
\end{figure}

\begin{figure}[htb]
\begin{center}
\includegraphics[scale=1]{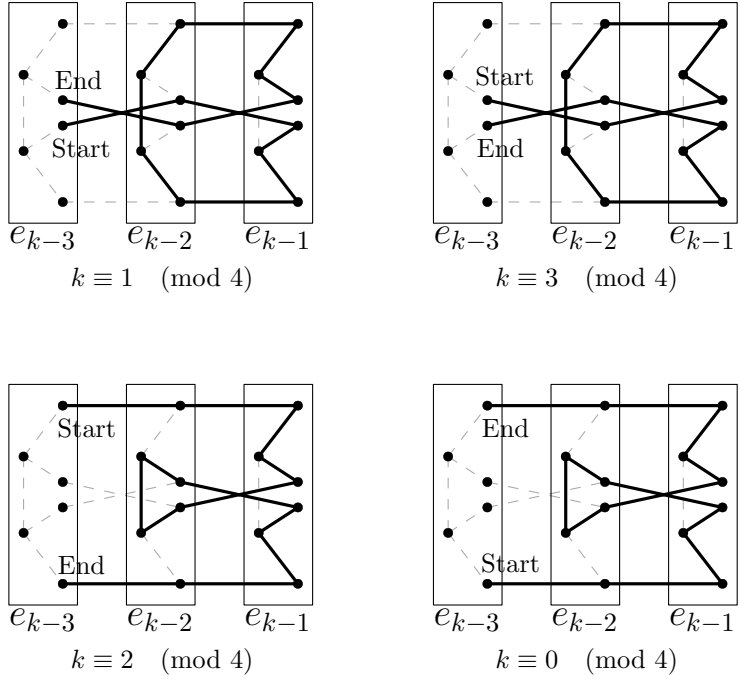}
\end{center}
\caption{The four possibilities for $P_{3}$}\label{18verts_directed_ham}
\end{figure}

Recall that because $v$ is even, $\l$ is also even. Let $R$ be the subgraph of the $2$-BIG of $\ts(v,\l)$ formed by the blocks from Steps $1c$, $2a$, $2b$, and $3$. The next lemma puts all of the pieces in this section together to show that $R$ contains a Hamilton path.

\begin{lemma}\label{hamPathSteps1c-3}
There is a Hamilton path in $R$ with endpoints 

\begin{itemize}
\item $A=\{\infty_{i_{1}},(a_2,j_1),(a_3,k_1)\}$ for any $a_2,a_3 \in \Z_{n}$ and $i_1,j_1,k_1 \in \{0,1\}$\\
and
\item $B=\{\infty_{i_{2}},(4,j_2),(n-2,k_2)\}$ for some $i_2,j_2,k_2 \in \{0,1\}$.
\end{itemize}
\end{lemma}

\begin{proof}
Let $G$, a subgraph of $R$, be the subgraph of the $2$-BIG of $\TS(v,\l)$ formed from the blocks in Steps~$1c$, $2a$, and $2b$.
Let $D_g'$ be a directed graph formed from Step 0 in Construction~\ref{sblocks} with vertex set $\mathbb{Z}_n$ such that there is a directed arc from $i+g$ to $-2i+g$ for each $i \in \mathbb{Z}_n$ where indices are calculated modulo $n$. Let $D'=\bigcup_{g\in \Z_{\l/2}} D_g'$.
Notice that the entire structure of the subgraph of the $2$-BIG of $\TS(v,\l)$ for Steps $1c$, $2a$, and $2b$ in Construction~\ref{sblocks} is dictated by the choice of which arcs in $D'$ are red or blue (see Step $0$ in Construction~\ref{sblocks}). 
We will form $G$ by using the graph $D'$, and we will handle the vertices in Step 3 of the  Construction~\ref{sblocks} by modifying the path we form in $G$.

By Lemma~\ref{digraphConnected}, the directed graph $D'$ is strongly connected, i.e. there is a directed walk from any vertex in $D'$ to any other vertex in $D'$. 
If $x \in \{g,g+1\}$, then $x$ has exactly one in-degree and one out-degree. Thus, since the in-degree of each vertex is the same as its out-degree and $D'$ is strongly connected, 
there is a directed Euler tour $\etour$ on the arcs in $D'$. 
The graph $G$ formed from $D'$ is connected by Lemma~\ref{sg+sg+1connected}, so we will use $\etour$ to find
a Hamilton path.
Each arc of $\mathbb{E}$ corresponds to a directed colored arc (red or blue) in $D'$ and this directed colored arc corresponds to six vertices in $G$ depending on Steps~$1c$, $2a$, or $2b$. 
Since $D_1'$ is a subgraph of $D'$, either $(4,n-2)$ or $(n-2,4)$ is an arc in $E(D')$.
Without loss of generality, let $\etour$ end on an arc $e_k\in \{(4,n-2),(n-2,4)\}\subseteq E(D')$ that represents $6$ vertices in $G$; one of these $6$ vertices in $G$ is vertex $B$. 
We know that we can have $e_k$ be the last arc in $\etour$ since we can think of the $\etour$ as the union of edge-disjoint directed closed trails all containing either the vertex $n-2$ or $4$, and we can just choose one arc of a closed trail to be the end of the Euler tour. 
This does mean that we forfeit some ability to choose on which arc our Euler tour begins. Fortunately, we will not simply follow the directed Euler tour when creating a Hamilton path in $G$.
Since $A$ exists, there exists an edge $(a_2,a_3)$ or $(a_3,a_2)$ in $D'$. 
Let $e_\ell$ be the arc in $\etour$ that represents six vertices, one of which is vertex $A$.

Partition the Euler tour $\etour$ into two or three directed trails: 
\begin{packedEnum}
\item[1)] $T_1=(e_1,\ldots,e_\ell)$ and $T_2=(e_{\ell+1},\ldots,e_k)$ if $\ell\neq 1$ and $|\{e_{\ell+1},e_{\ell+2},\ldots,e_k\}|$ is even,
\item[2)] $T_1=(e_{2},\ldots,e_{\ell})$, $T_2=(e_{\ell+1},e_{\ell+2},\ldots,e_{k-1})$, and $T_3=(e_k,e_1)$ if $\ell\neq 1$ and $|\{e_{\ell+1},e_{\ell+2},\ldots,e_k\}|$ is odd, and 
\item[3)] $T_1=(e_1,e_2,\ldots,e_{k-2})$ and $T_2=(e_k,e_{k-1})$ if $\ell=1$.
\end{packedEnum}
\textbf{Case 1:} Suppose that $\ell\neq 1$ and  $|\{e_{\ell+1},e_{\ell+2},\ldots,e_k\}|$ is even. By Lemma~\ref{part_etrail_ham}, there exists a path $P_1$ in the subgraph of $G$ represented by $T_1$ that begins at $A$ (one of the four vertices represented by $e_\ell$ in Step $2a$ or $2b$), and ends at 
some other vertex $A'$ represented by $e_{\ell}$.
By Lemma~\ref{etrail_ham} and since the length of $T_2$ is even, we can pair consecutive arcs (e.g., $e_{\ell+1}$ and $e_{\ell+2}$) in $T_2$ and find a path $P_2$ in the subgraph of $G$ that begins at a vertex adjacent to $A'$ and ends at the vertex $B$.
Thus we can form the intended Hamilton path.

\noindent\textbf{Case 2:} Suppose that $\ell\neq 1$ and  $|\{e_{\ell+1},e_{\ell+2},\ldots,e_k\}|$ is odd; then $|T_2|$ is even. By Lemma~\ref{part_etrail_ham}, there exists a path $P_1$ in the subgraph of $G$ represented by $T_1$ that begins at $A$ (one of the four vertices represented by $e_\ell$ in Step $2a$ or $2b$), and ends at 
some other vertex $A'$ represented by $e_{\ell}$.
By Lemma~\ref{etrail_ham} and since the length of $T_2$ is even, we can pair consecutive arcs in $T_2$ and find a path $P_2$ in the subgraph of $G$ that begins at a vertex adjacent to $A'$ and ends at a vertex $B'$ in $e_{k-1}$. Then by Lemma~\ref{part_etrail_ham}, there exists a path $P_3$ in the subgraph of $G$ represented by $T_3$ that begins at a vertex (one of the four vertices represented by $e_k$ in Step 2a or 2b) adjacent to $B'$ and ends at the vertex $B$ (one of the four vertices represented by $e_k$ in Step 2a or 2b). 
Note that $B$ only needs to be one of the four vertices represented by $e_k$ in Step 2a or 2b. So there is no need to specify which one is $B$ until we are forced to choose, as may happen in this case.

\noindent\textbf{Case 3:} Suppose that $\ell=1$. 
By Lemma~\ref{part_etrail_ham}, there exists a path $P_1$ in $G$ that begins at $A$ (one of the four vertices represented by $e_\ell=e_1$ in Step $2a$ or $2b$) and ends at another one of the four vertices in $G$, say $A'$ represented by $e_\ell$ in Step $2a$ or $2b$. By Lemma~\ref{part_etrail_ham}, there exists a path $P_2$ in $G$ that begins at a vertex in $G$ represented by $e_k$ which is adjacent to $A'$ and ends at another one of the four vertices in $G$ represented by $e_k$, which is $B$.

In each of the above cases, either $P=P_1\circ P_2$ or $P=P_1\circ P_2\circ P_3$ forms a Hamilton path through the subgraph of $G$ represented by Steps $1c$, $2a$, and $2b$.

It remains to augment $P$ to form a Hamilton path that contains all of the vertices from the subgraph of the $2$-BIG of $\ts(v,\l)$ formed from the blocks in Steps $1c$, $2a$, $2b$, and $3$. For each $g\in \Z_{\l/2}$, a $K_4$ is formed from the vertices in Step $3$. There is an arc $e_i=(g,b)$ in $\etour$ (since $D'$ is stongly connnected) for some $b\in V(D')$ that represents vertices in $G$ that are adjacent to the vertices in Step 3. Each of the $6$ vertices represented by the arc $e_i$ are adjacent to at least two vertices in this $K_4$. Delete some edge in $P$ between two of these $6$ vertices, say the edge joining $\{\infty_{x_3},(g,y_3),(b,z_3)\}$ to $\{(g,x_4),(b,y_4),(g,z_4)\}$ where $x_3,x_4,y_3,y_4,z_3,z_4\in \{0,1\}$, and add a path through all of the vertices in this $K_4$ with endpoints at the two ends of the edge deleted from $P$. Do this for each $K_4$ formed from Step $3$ to form the required Hamilton path. 
\end{proof}

\section{Conclusion}\label{sec:conclusion}

In this section we prove Theorem~\ref{v04mod12}. 
We will use the following observation, which is illustrated in Figure~\ref{lem_1_pics}.

\begin{observation}
\label{cube}
For any vertex $v_{1} \in V(Q_3)$, there are four different vertices at distance $1$ or $3$ from $v_1$ that we could choose as $v_2$ such that there is a Hamilton path through $Q_3$ with endpoints $v_{1}$ and $v_{2}$.
\end{observation}

\begin{figure}[htb]
\begin{center}
\includegraphics[scale=.75]{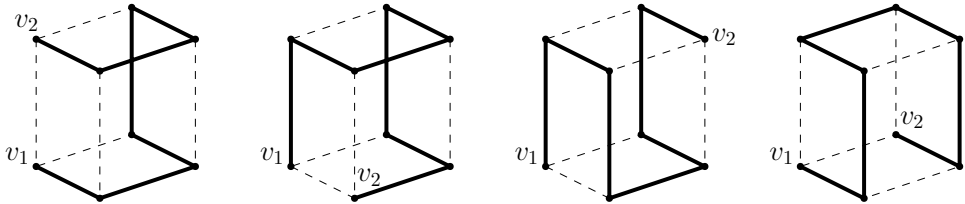}
\end{center}
\caption{Four Hamilton path from $v_1$ to $v_2$}\label{lem_1_pics}
\end{figure}


\setcounter{theorem}{0}
\begin{theorem}
If $v\equiv 0$ or $4\pmod{12}$, then for all admissible $v,\l$ there exists a simple $\TS(v,\l)$ with a cyclic $2$-intersecting Gray code.
\end{theorem}

\begin{proof}
Let $G$ be the $2$-BIG of the $\TS(v,\l)$ formed from Construction~\ref{sblocks}. 
By Lemma~\ref{HamProduct}, there is a Hamilton cycle through the subgraph $W \square Q_{3}$. Furthermore, because $\{Z,Z'\} \in E(W^{n}_{0,1})$ for some for some $Z'\in H_{0,1}^n$,
it follows that we may remove an edge from the Hamilton cycle with endpoint $x$, a vertex in the cube corresponding to $Z'$. Then $x$ is of
the form

\[x=\{(X_1,i_0),(X_2,j_0),(X_3,k_0)\}\]
where $\{X_1,X_2,X_3\}$ contains one of the pairs $\{1,4\}$, $\{1,n-2\}$, or $\{4,n-2\}$. Let $x'=\{(X_{1}',i_1),(X_{2}',j_1),(X_{3}',k_1)\}$ be
the other endpoint of the removed edge. The result is a Hamilton $(x,x')$-path, $P_{1}$.

By Lemma~\ref{hamPathSteps1c-3}, there is a Hamilton path, $P_{2}$ through the subgraph $R$ with endpoints

\[y=\{\infty_{i_2}, (X_{2}',j_2),(X_{3}',k_2)\} \mbox{ for any } X'_2, X'_3 \in \Z_n \mbox{ and }i_2,j_2,k_2 \in \{0,1\}\]
\[y'=\{\infty_{i_3}, (4,j_3),(n-2,k_3)\} \mbox{ for some }i_3,j_3,k_3 \in \{0,1\}.\]

\noindent Choose $j_2=j_1$ and $k_2=k_1$ so $y$ is adjacent to $x'$.

There is a Hamilton path, $P_3$, through the cube corresponding to $Z \in V(H^{n}_{0,1, \ldots,\frac{\l}{2}-1})$, by Observation~\ref{cube}. If we choose one
of the endpoints to be $z=\{(1,i_4),(4,j_4),(n-2,k_4)\}$ with $j_4=j_3$ and $k_4=k_3$, then $z$ is adjacent to $y'$.
We may assume without loss of generality, that $z=\{(1,0), (4,0), (n-2,0)\}$. Then by Observation~\ref{cube}, the other endpoint of $P_{3}$, $z'$, can
be chosen from the set $\{\{(1,0),(4,1),(n-2),1)\}, \{(1,1),(4,1),(n-2),0)\}, \{(1,0),(4,0),(n-2),0)\}, \{(1,1),(4,0),(n-2),1)\}\}$.
Thus we may choose $z'$ so that it is adjacent to $x$. Then $C=P_1 \circ P_2 \circ P_3$ is a Hamilton cycle.

\end{proof}

Our future work will use Theorem~\ref{v04mod12} as a base case to show that there exists a $\ts(v,\l)$ whose $2$-BIG is Hamiltonian when $v$ is even. 

\appendix
\section{Appendix}
\label{sec:appendix}

\setcounter{theorem}{0}
\renewcommand{\thetheorem}{\Alph{section}.\arabic{theorem}}

\begin{lemma}\label{appendix:subgraph}
$H^{n-6}_{g-1,g}$ is a subgraph of  $H^n_{g-1,g}$ for any $n$.
\end{lemma}

\begin{proof}
First, we show that each $U_{i,j}$ is unique.
Elements in $U_{i_1,j_1}$ sum to $3g-3$ or $3g$ when $j_1\equiv 0\pmod{2}$ and $n\equiv 1$ or $5\pmod{n}$ respectively, while $U_{i_2,j_2}$ sums to $3g$ or $3g-3$ when $j_2\equiv 1\pmod{2}$ and $n\equiv 1$ or $5\pmod{n}$. Therefore, $U_{i,j_1}\neq U_{i,j_2}$ when $j_1\not\equiv j_2 \pmod{2}$. 
It remains to show $U_{i_1,j}\neq U_{i_2,j}$ when $i_1\neq i_2$. In order
to show this, we assume that they are equal. This can only happen if there is equivalence modulo $n$ between each of the points in $U_{i_1,j}$ and $U_{i_2,j}$. We will show that in each case, a contradiction arises.
Let $j\equiv 1\pmod{2}$ and $n\equiv 1\pmod{6}$ so that
\begin{align*}
U_{i_1,j} &= \left\{g+3i_1+5, \frac{1}{2}(2g-3j+6i_1+7),\frac{1}{2} (2g+3j-12i_1-23)\right\}_{g-1}=:\{\alpha_1,\beta_1,\gamma_1\}, \text{ and} \\
U_{i_2,j} &= \left\{g+3i_2+5, \frac{1}{2}(2g-3j+6i_2+7),\frac{1}{2} (2g+3j-12i_2-23)\right\}_{g}=:\{\alpha_2,\beta_2,\gamma_2\}.
\end{align*}
Notice that if $\alpha_1\equiv \alpha_2\pmod{n}$, $\beta_1\equiv \beta_2\pmod{n}$, or $\gamma_1\equiv \gamma_2\pmod{n}$, then $i_1-i_2\equiv 0\pmod{n}$. In fact, this previous statement is true regardless of whether $n\equiv 1$ or $5\pmod{6}$, or $j\equiv 0$ or $1\pmod{2}$. But this implies $i_1-i_2\equiv 0\pmod{n}$ and we assumed this cannot occur. 
So we may assume that $\alpha_1,\beta_1,\gamma_1$ is not equivalent to $\alpha_2,\beta_2,\gamma_2$ modulo $n$ respectively when $n\equiv 1$ or $5\pmod{6}$ or $j\equiv 0$ or $1 \pmod{2}$. 
Suppose that $\alpha_1\equiv \beta_2\pmod{n}$, $\beta_1\equiv \gamma_2\pmod{n}$, and $\gamma_1\equiv \alpha_2\pmod{n}$. Then we have $3(i_1-i_2+\frac{j+1}{2})\equiv 0\pmod{n}$, $3(i_1+2i_2-j+5)\equiv 0\pmod{n}$, and $-3(2i_1+i_2-\frac{j-11}{2})\equiv 0\pmod{n}$. 
The first and third equivalence relations tell us $i_1\equiv -2\pmod{n}$, which tell us $i_1 \geq n-2$, but we have defined $i_1 \leq \frac{n-7}{3}$, so we have a contradiction.
The only other alternative is if $\alpha_1\equiv \gamma_2\pmod{n}$, $\beta_1\equiv \alpha_2\pmod{n}$, and $\gamma_1\equiv \beta_2\pmod{n}$, and similar argument shows a contradiction as well. Thus $U_{i_1,j}\neq U_{i_2,j}$ for $j\equiv 1\pmod{2}$ and $n\equiv 1\pmod{6}$.

Now suppose $n\equiv 1\pmod{6}$ and $j\equiv 0\pmod{2}$. Then if we again assume $\alpha_1\equiv \beta_2\pmod{n}$, $\beta_1\equiv \gamma_2\pmod{n}$, and $\gamma_1\equiv \alpha_2\pmod{n}$, we get
$\frac{3}{2}(2i_1-2i_2+j)\equiv 0\pmod{n}$, $3(5+i_1+2i_2-j)\equiv 0\pmod{n}$, and $-\frac{3}{2}(10+4i_1+2i_2-j)\equiv 0\pmod{n}$.  
The first and third equivalences tell us that $3i_1\equiv-5\pmod{n}$, but then $i_1\geq\frac{n-5}{3}$ and we defined $i_1\leq \frac{n-8}{3}$, a contradiction. 
Thus $U_{i_1,j}\neq U_{i_2,j}$ for $j\equiv 0\pmod{2}$ and $n\equiv 1\pmod{6}$.

Suppose that $n\equiv 5\pmod{6}$ and $j\equiv 1\pmod{2}$. Then using the same strategy as above, we get that 
$3(6+2i_1+i_2-j)\equiv 0\pmod{n}$, $-\frac{3}{2}(-1+2i_1-2i_2-j)\equiv 0\pmod{n}$, and $-\frac{3}{2}(13+2i_1+4i_2-j)\equiv 0\pmod{n}$. 
The last two equivalences tell us that $3i_2\equiv -7\pmod{n}$, but we assumed that $i_2\leq \frac{n-8}{3}$, a contradiction.
Thus $U_{i_1,j}\neq U_{i_2,j}$ for $j\equiv 1\pmod{2}$ and $n\equiv 5\pmod{6}$.

Suppose that $n\equiv 5\pmod{6}$ and $j\equiv 0\pmod{2}$. Then using the same strategy as above, we get that 
$3(6+2i_1+i_2-j)\equiv 0\pmod{n}$, $-\frac{3}{2}(2i_1-2i_2-j)\equiv 0\pmod{n}$, and $-\frac{3}{2}(12+2i_1+4i_2-j)\equiv 0\pmod{n}$.
The last two equivalences tell us that $i_2\equiv -2$, but this is impossible since $i_2\leq \frac{n-8}{3}$, a contradiction. In all cases, if
$\alpha_1\equiv \gamma_2\pmod{n}$, $\beta_1\equiv \alpha_2\pmod{n}$, and $\gamma_1\equiv \beta_2\pmod{n}$, a similar contradiction arises.
Thus $U_{i_1,j}\neq U_{i_2,j}$. 
The same methods can be used to show that $U_{i_1,j}\neq \overline{U}_{i_2,j}$ and $\overline{U}_{i_1,j}\neq \overline{U}_{i_2,j}$. 

We next show that in $H^n_{g-1,g}$ the following holds: 
\begin{packedItem}
\item $U_{i,j}\sim U_{i,j+1}$, 
\item $U_{i,j}\sim U_{i+1,j+3}$ if $j$ is odd and $n\equiv 1\pmod{6}$, or if $j$ is even and $n\equiv 5\pmod{6}$, 
\item $\overline{U}_{i,j}\sim \overline{U}_{i,j+1}$, 
\item $\overline{U}_{i,j}\sim \overline{U}_{i+1,j+3}$ if $j$ is odd, and 
\item either $U_{(n-7)/3,j}\sim \overline{U}_{(n-7)/3,j}$ if $j\equiv 1\pmod{2}$ and $n\equiv 1\pmod{6}$ or  $U_{(n-8)/3,j}\sim \overline{U}_{(n-8)/3,j}$ if $j\equiv 1\pmod{2}$ and $n\equiv 5\pmod{6}$. 
\end{packedItem}
Since the parity of $j$ is different between $U_{i,j}$ and $U_{i,j+1}$, it is an easy calculation to verify that $U_{i,j}$ and $U_{i,j+1}$ are adjacent when $n\equiv 1$ or $5\pmod{6}$ by examining the definition of $U_{i,j}$. This is true for both $n\equiv 1\pmod{6}$ and $n\equiv 5\pmod{6}$. 

Suppose that $j\equiv 1\pmod{2}$ and $n\equiv 1\pmod{6}$. Then 
\begin{align*}
U_{i+1,j+3}&=\left\{g+3(i+1)+5,\frac{1}{2}(2g-3(j+3)+6(i_1+1)+10),\right.\\
&\phantom{\;\;\;\;\;}\left.\frac{1}{2} (2g+3(j+3)-12(i_1+1)-20)\right\}_{g}\\
&= \left\{g+3i+8,\frac{1}{2}(2g-3j+6i+7),\frac{1}{2}(2g+3j-12i-23)\right\}_{g}.
\end{align*}
Thus $U_{i,j}\sim U_{i+1,j+3}$. The same calculation will show that $U_{i,j}\sim U_{i+1,j+3}$ when $j\equiv 0\pmod{2}$ and $n\equiv 5\pmod{6}$. 
Since $\overline{U}_{i,j}=\{a',b',c'\}$ is defined by $U_{i,j}=\{a,b,c\}$ where $a'=2g-1-a$, $b'=2g-1-b$, and $c'=2g-1-c$, it is clear that $\overline{U}_{i,j}\sim \overline{U}_{i,j+1}$ and $\overline{U}_{i,j}\sim\overline{U}_{i+1,j+3}$.
It remains to show that when $j$ is odd either $U_{(n-7)/3,j}\sim \overline{U}_{(n-7)/3,j}$ or  $U_{(n-8)/3,j}\sim \overline{U}_{(n-8)/3,j}$ if $n\equiv 1\pmod{6}$ or if $n\equiv 5\pmod{6}$ respectively. 
Suppose that $n\equiv 1\pmod{6}$ and $j\equiv 1\pmod{2}$. Then 
\begin{align*}
U_{(n-7)/3,j} &= \left\{g+n-2,\frac{1}{2}(2g-3j+2n-7),\frac{1}{2}(2g+3j-4n+5)\right\}_{g-1}
\end{align*}
and
\begin{align*}
\overline{U}_{(n-7)/3,j} &= \left\{g-n+1,\frac{1}{2}(2g+3j-2n+5),\frac{1}{2}(2g-3j+4n-7)\right\}_{g}.
\end{align*}
Since the points in the triples are calculated modulo $n$, it is clear that $U_{(n-7)/3,j}\sim\overline{U}_{(n-7)/3,j}$.
The argument will be similar for $n\equiv 5\pmod{6}$, thus it is omitted.

Based on the adjacencies given above, it is clear that $H^{n-6}_{g-1,g}$ is a subgraph of  $H^n_{g-1,g}$ for any $n$.
\end{proof}

 \bibliographystyle{amsplain}
\bibliography{vdec2}

\end{document}